\documentclass[11pt,a4paper]{article}
\usepackage[utf8]{inputenc}
\usepackage{amsmath}
\usepackage{amsthm}
\usepackage{amsfonts}
\usepackage{multirow}
\usepackage[table,xcdraw]{xcolor}
\usepackage{subcaption}
\usepackage{float}
\usepackage{rotating}
\usepackage{hyperref}
\usepackage{tabularx}
\usepackage{multicol}
\usepackage{amssymb}
\usepackage{blindtext}
\usepackage{graphicx}
\usepackage[shortlabels]{enumitem}
\usepackage[left=2cm,right=2cm,top=2cm,bottom=2cm]{geometry}
\usepackage{booktabs}
\usepackage[table,xcdraw]{xcolor}
\usepackage{multirow}
\usepackage{rotating,tabularx}
\usepackage{longfbox}

\newtheorem{Step}{Step}
\newtheorem{Theorem}{Theorem}[section]
\newtheorem{Proposition}[Theorem]{Proposition}
\newtheorem{Remark}[Theorem]{Remark}
\newtheorem{Lemma}[Theorem]{Lemma}
\newtheorem{Corollary}[Theorem]{Corollary}
\newtheorem{Definition}[Theorem]{Definition}

\newtheorem{Algorithm}{Algorithm}
\usepackage{hyperref}
\hypersetup{colorlinks=true,urlcolor=red}
\expandafter\let\expandafter\oldproof\csname\string\proof\endcsname
\let\oldendproof\endproof
\renewenvironment{proof}[1][\proofname]{
\oldproof[\ttfamily\scshape \bf #1.]
}{\oldendproof}
\newcommand{\set}[1]{\left\{#1\right\}}
\def\ve{\varepsilon}

\def\tilde{\widetilde}
\def\emp{\emptyset}

\def\B{\mathbb B}

\def\ox{\overline{x}}

\def\Hat{\widehat}

\def\Bar{\overline}
\def\ra{\rangle}
\def\la{\langle}
\def\ve{\varepsilon}
\def\epsilon{\varepsilon}
\def\ox{\bar{x}}

\def\prox{\mbox{\rm Prox}}

\def\dn{\downarrow}
\def\O{\Omega}
\def\ph{\varphi}
\def\emp{\emptyset}
\def\st{\stackrel}

\def \N{{\rm I\!N}}
\def \R{{\rm I\!R}}
\def\th{\theta}

\newcommand{\dotproduct}[1]{\left\langle#1\right\rangle}
\newcommand{\brac}[1]{\left(#1\right)}
\newcommand{\sbrac}[1]{\left[#1\right]}
\newcommand{\abs}[1]{\left|#1\right|}
\newcommand{\norm}[1]{\left\|#1\right\|}
\numberwithin{equation}{section}

\title{\bf Inexact reduced gradient methods\\ in nonconvex optimization}
\author{Pham Duy Khanh\footnote{Department of Mathematics, Ho Chi Minh City University of Education, Ho Chi Minh City, Vietnam. E-mail: pdkhanh182@gmail.com}\quad Boris S. Mordukhovich\footnote{Department of Mathematics, Wayne State University, Detroit, Michigan, USA. E-mail: aa1086@wayne.edu. Research of this author was partly supported by the US National Science Foundation under grants DMS-1808978 and DMS-2204519, by the Australian Research Council under grant DP-190100555, and by Project 111 of China under grant D21024.}\quad Dat Ba Tran\footnote{Department of Mathematics, Wayne State University, Detroit, Michigan, USA. E-mail: tranbadat@wayne.edu. Research of this author was partly supported by the US National Science Foundation under grant DMS-1808978.}}
\begin{document}
\maketitle

\noindent
{\small{\bf Abstract}. This paper proposes and develops new linesearch methods with inexact gradient information for finding stationary points of nonconvex continuously differentiable functions on finite-dimensional spaces. Some abstract convergence results for a broad class of linesearch methods are established. A general scheme for inexact reduced gradient (IRG) methods is proposed, where the  errors in the gradient approximation automatically adapt with the magnitudes of the exact gradients. The sequences of iterations are shown to obtain stationary accumulation points when different stepsize selections are employed. Convergence results with constructive convergence rates for the developed IRG methods are established under the Kurdyka-\L ojasiewicz property. The obtained results for the IRG methods are confirmed by encouraging numerical experiments, which demonstrate advantages of automatically controlled errors in IRG methods over other frequently used error selections.}\\[1ex] 
\vspace*{-0.15in}

 \section{Introduction}\label{sec:intro}\vspace*{-0.05in}
 
Consider the unconstrained optimization problem  formulated as follows:
\begin{align}\label{optim prob}
{\rm minimize}\quad f(x)\quad\text{ subject to }\; x\in\R^n
\end{align}
with a continuously differentiable ($\mathcal{C}^1$-smooth) objective function $f:\R^n\rightarrow\R$. One of the most natural and classical approaches to solve (\ref{optim prob}) is by using \textit{linesearch methods}; see, e.g., \cite{beck,bertsekasbook,solodovbook,nesterovbook18,nocedalbook,polyakbook}. Given a starting point $x^1\in\R^n$, such methods construct the iterative procedure 
\begin{align}\label{intro: stepsize + descent direction}
x^{k+1}:=x^k+t_kd^k\;\text{ for all }\;k\in\N,
\end{align}
where $t_k\ge 0$ is a stepsize at the $k^{\rm th}$ iteration, and where the direction $d^k$ satisfies the  
condition
\begin{align*}
\big\la d^k,\nabla f(x^k)\big\ra<0.
\end{align*}
The classical choice for the direction is $d^k=-\nabla f(x^k)$ when the resulting algorithm is known as the {\em gradient descent method}; see the aforementioned books and the references therein. If $f$ is twice continuously differentiable ($\mathcal{C}^2$-smooth) and the Hessian matrix $\nabla^2 f(x^k)$ is positive-definite, then $d^k$ is chosen by solving the linear equation 
\begin{align*}
-\nabla f(x^k)=\nabla^2 f(x^k)d^k,
\end{align*}
and it is known as a {\em Newton direction} \cite{bertsekasbook,pangbook2,solodovbook}. Additionally, more general choices of descent directions  widely used are the \textit{gradient related} directions \cite[Page~41]{bertsekasbook}, \textit{directions satisfying an angle condition} \cite[Page~541]{absil05}, etc. Together with the descent directions, stepsizes are usually chosen to ensure the decreasing property of the entire sequence $\set{f(x^k)}$ or sometimes only its tail. Well-known stepsize selections are \textit{constant stepsize}, \textit{diminishing stepsize} (not summable), stepsizes following \textit{Armijo rule}, and \textit{Wolfe conditions}; see, e.g., \cite{absil05,beck,bertsekasbook,solodovbook,nesterovbook18,nocedalbook,polyakbook}. 

The stationarity of accumulation points generated by linesearch methods with \textit{gradient related} directions and stepsizes following the \textit{Armijo rule} is established in \cite[Proposition~1.2.1]{bertsekasbook}. When the Lipschitz continuity of the gradient is additionally assumed, the same type of convergence is achieved if either the 
stepsize is constant and directions are gradient related, or the stepsize is diminishing and directions satisfy more involved conditions \cite[Proposition~1.2.3]{bertsekasbook}. The global convergence of some linesearch methods to an isolated stationary point relies on the {\em Ostrowski condition}; see \cite{ostrowski} and \cite[Theorem~8.3.9]{pangbook2}. If there is no guarantee on the isolation of stationary points, algebraic geometry tools introduced by \L ojasiewicz \cite{lojasiewicz65} and Kurdyka \cite{kurdyka} are used for linesearch methods with directions satisfying the {\em angle conditions} and the stepsize following the \textit{Wolfe conditions}; see \cite[Theorem~4.1]{absil05}. While rates of convergence of general linesearch methods are not considered, some specific methods achieve certain convergence rates for particular classes of functions. For instance, the gradient descent method achieves a local linear rate of convergence if the objective function is twice differentiable with Lipschitz continuous Hessian as in \cite[Theorem~1.2.4]{nesterovbook18}, and the (generalized) damped Newton method attains a superlinear convergence rate of under the positive-definiteness of the (generalized) Hessian and some additional assumptions; \cite[Theorem~4.5]{kmpt21}. Furthermore, a linear rate of convergence for the gradient descent method is achieved under either the {\em  Polyak-\L ojasiewicz condition} as in \cite{polyak63} and \cite{karimi16}, or under the {\em weak convexity} of the objective function as in \cite{rotaru22}.

Due to its simplicity, the gradient descent method is broadly used to solve various optimization problems; see, e.g., \cite{bottou18,crockett55,curry44,nielsen15,ruder16}. However, errors in gradient calculations may appear in many situations, which can be found in practical problems arising, e.g., in the design of space trajectories \cite{addis11} and computer-aided design of microwave devices \cite{gilmore95}. Moreover, many nonsmooth optimization problems can be transformed into its smoothed versions by using Moreau envelopes \cite{rockafellarbook} and forward-backward envelopes \cite{stp}. Nevertheless, gradients of smoothed functions cannot be usually computed precisely, and therefore various gradient methods with {\em inexact gradient information} have been suggested. We mention the following major developments in this vein:\vspace*{-0.1in}
\begin{itemize}
\item Devolder et al. \cite{nesterov14} introduce the notion of {\em inexact oracle} and analyze behavior of several first-order methods of {\em smooth convex} optimization employed such an oracle. Nesterov \cite{nesterov15} develops new methods of this type for {\em nonsmooth convex} optimization in the framework of inexact oracle.\vspace*{-0.05in}

\item Gilmore and Kelley \cite{gilmore95} propose an {\em implicit filtering algorithm} to deal with certain box constrained optimization problems, where the objective function is a sum of a $\mathcal{C}^1$-smooth function with Lipschitz continuous gradient and a noise function.\vspace*{-0.05in} 

\item Bertsekas shows in \cite[pp.~44--45]{bertsekasbook} that if the objective function is $\mathcal{C}^1$-smooth with a Lipschitz continuous gradient and if the error of inexact gradient is either small relative to the norm of the exact gradient, or proportional to the stepsize, then convergence behavior of gradient methods is similar to the case where there are no errors.
\end{itemize}\vspace*{-0.05in}

All the convergence results for inexact gradient methods mentioned above  assume that the objective function is {\em either $\mathcal{C}^1$-smooth} with a Lipschitz continuous gradient, {\em or convex}. To the best of our knowledge, general methods of solving nonconvex $\mathcal{C}^1$-smooth optimization problems with inexact information on non-Lipschitzian gradients are {\em not available} in the literature. One of the reasons for this is that verifying the descent property of the sequence of function values without the Lipschitz continuity of $\nabla f$ via the Armijo linesearch requires exact information on gradients. To deal with  inexactness, we need a descent direction that allows us to {\em replace the Armijo linesearch procedure} by another one not demanding exact gradients. In addition, a practical inexact gradient method that uses constant stepsize for a general nonconvex $\mathcal{C}^1$-smooth function with the Lipschitz gradient is also not established. Although an inexact gradient method with constant stepsize is proposed in \cite[pp.~44--45]{bertsekasbook} by using an error smaller than the norm of the exact gradient, the problem of how to control this error while the exact gradient is unknown is still questionable.

Having in mind the above discussion, we introduce new \textit{inexact reduced gradient} (IRG) methods to find stationary points for a general class of nonconvex $\mathcal{C}^1$-smooth functions. Although our proposed methods address {\em smooth} problems, some motivations for them partly come from a certain {\em nonsmooth} algorithm and {\em generalized differential} tools of variational analysis. Specifically, to find a Clarke stationary point of a nonsmooth locally Lipschitzian function, the {\em gradient sampling} (GS) method, introduced by Burke et al. \cite{blo02} and modified by Kiwiel \cite{kiwiel07}, approximates at each iteration the $\ve$-{\em generalized gradient} by the convex hull of nearby gradients. In the GS method, the negative projection of the origin onto this convex hull is chosen as the descent direction and the stepsizes are chosen from the backtracking linesearch as in \cite[Section~4.1]{kiwiel07}. Although the GS method work well for nonsmooth problems, using them for smooth functions seems to be challenging due to, in particular, the necessity to solve subproblems of finding projections onto convex hulls. However, replacing the $\varepsilon$-generalized gradient  by the Fr\'echet-type $\varepsilon$-{\em subdifferential} makes our methods much simpler and suitable for smooth problems. Indeed, the latter construction for a $\mathcal{C}^1$-smooth function at the point in question is just the closed ball centered at its gradient with radius $\varepsilon$. Thus the projection of the origin onto this ball has an explicit and simple form. The descent direction chosen by this projection also allows us to replace the exact gradient by its approximation and to use a linesearch procedure that does not require exact gradients. Developing this idea, we design our \textit{inexact reduced gradient} methods with non-normalized directions together with some stepsize selections such as backtracking stepsize, constant stepsize, and diminishing stepsize. To the best of our knowledge, the IRG methods that we propose and develop in this paper are {\em completely new} even in the {\em exact case}. It should also be emphasized that the proposed IRG methods are {\em not special versions} of the GS one since the latter needs exact gradients at multiple points in each iteration, while the IRG methods need only {\em one inexact  gradient}. Moreover, the iterative sequence of the GS method is chosen randomly while IRG iterations are designed {\em deterministically}. Our main results include the following:
\vspace*{-0.1in}

\begin{itemize}
\item Designing a {\em general framework} for the IRG methods and revealing their basic properties.\vspace*{-0.05in}

\item Finding {\em stationary accumulation points} of iterations in the IRG methods with  backtracking stepsizes as well as with either {\em constant stepsizes}, or {\em diminishing ones} under an additional descent condition on the objective functions.\vspace*{-0.05in}

\item Obtaining the {\em global convergence} of iterations in the IRG methods with {\em constructive convergence rates} depending on the exponent of the imposed {\em Kurdyka-\L ojasiewicz $(KL)$ property} of the objective functions. 
\end{itemize}\vspace*{-0.1in}

These results are achieved by using our newly developed scheme for general linesearch methods described in the following way. To begin with, some conditions are proposed to ensure the stationary of accumulation points in general linesearch methods. If the KL property is additionally assumed, then the global convergence of the iterative sequence to a stationary point is guaranteed. Moreover, the rates of convergence are established if the stepsize is bounded away from zero.\vspace*{0.03in}

From a practical viewpoint, our IRG methods automatically adjust the errors required for finding approximate gradients, which will be shown to have numerical advantages over decreasing errors, e.g., $\varepsilon_k=k^{-p}$ as $p\ge 1$, that are frequently used in the existing methods \cite{bertsekasbook,nesterov14,gilmore95}. To elaborate more on this issue, observe that since the magnitude of the exact gradient is small near the stationary points and is larger elsewhere, decreasing errors that do not take the information of the exact or inexact gradients into consideration may encounter the following phenomena:\vspace*{-0.1in}
\begin{itemize}
\item\textit{Over approximation}, which happens when the magnitude of the exact gradient is large but the error is too small. In this case, the procedures of finding approximate gradient may execute longer than needed to obtain a good approximation of the exact gradient.\vspace*{-0.05in}

\item\textit{Under approximation}, which happens when the magnitude of the  exact gradient is small but the error is too large, which may lead us to an approximate gradient that is not good enough. As a consequence of using such an approximate gradient, the next iterative element can be worse instead of being better than the current one.
\end{itemize}\vspace*{-0.1in}

In contrast to methods using decreasing errors,  our IRG methods, by performing a low cost checking step in each iteration to determine whether the error for the approximation procedure should be decrease or stay the same in the next iteration, use errors which automatically adapt with the magnitudes of the exact gradients to avoid the aforementioned phenomena and exhibit a better performance.\vspace*{0.05in}

The rest of the paper is organized as follows. Section~\ref{sec prelim} discusses basic notions related to  methods. Some abstract convergence results for linesearch methods are established in Section~\ref{sec auxiliary}. In Section~\ref{sec Stationarity of accumulation points}, we introduce a general form of the IRG methods and investigate their principal properties. Our main results about the convergence behavior of the IRG methods with different stepsize selections are given in Section~\ref{sec 5}. The numerical experiments conducted in Section~\ref{sec: Numerical experiments} support the theoretical results obtained in Section~\ref{sec 5} and show that the IRG methods with the new type of automatically controlled errors have a better performance in comparison with the inexact proximal point method in the Least Absolute Deviations Curve-Fitting problem taken from \cite{peter80}. In Section~\ref{sec: Numerical experiments}, we also compare numerically the performance of our IRG methods with those of the reduced gradient method and the gradient decent method employing the exact gradient calculations for two well-known benchmark functions in global optimization. The last Section~\ref{sec:conc} discusses some directions of our future research.
\vspace*{-0.15in}

\section{Linesearch Methods and Related Properties}\label{sec prelim}\vspace*{-0.05in}

First we recall some notions and notation frequently used in what follows. All our considerations are given in the space $\R^n$ with the Euclidean norm $\|\cdot\|$ and scalar/inner product $\dotproduct{\cdot,\cdot}$. We use $\N:=\{1,2,\ldots\}$, $\R_+$, and $\overline{\R}:=\R\cup\set{\infty}$ to denote the collections of natural numbers, positive numbers, and the extended real line, respectively. The symbol $x^k\st{J}{\to}\ox$ means that $x^k\to\ox$ as $k\to\infty$ with $k\in J\subset\N$. For a $\mathcal{C}^1$-smooth function  $f:\R^n\rightarrow\R$, $\bar x$ is a \textit{stationary point} of $f$ if $\nabla f(\bar x)=0$. The function $f$ is said to satisfy the {\em $L$-descent condition} with some $L>0$ if
\begin{align}\label{descent condition}
f(y)\le f(x)+\dotproduct{\nabla f(x),y-x}+\dfrac{L}{2}\norm{y-x}^2\;\text{ for all }\;x,y\in\R^n.
\end{align}
We see that $L$-descent condition \eqref{descent condition} means that the graphs o the quadratic functions $f_{L,x}(y):=f(x)+\dotproduct{\nabla f(x),y-x}+\frac{L}{2}\norm{y-x}^2$ lie above that of $f$ for all $x\in\R^n.$ This condition is equivalent to the convexity of $\dfrac{L}{2}\norm{x}^2-f(x)$ \cite[Lemma~4]{xingyu18}, while being a direct consequence of the $L$-Lipschitz continuity of $\nabla f$, i.e.,\ the Lipschitz continuity of $\nabla f$ with constant $L$; see, e.g., \cite[Proposition~A.24]{bertsekasbook} and \cite[Lemma~A.11]{solodovbook}. The converse implication holds when $f$ is convex \cite{baus17,xingyu18} but fails otherwise. A major class of real-valued functions satisfying the $L$-descent condition but not having the Lipschitz continuous gradient is given by
\begin{align*}
f(x):=\dfrac{1}{2}\dotproduct{Ax,x}+\dotproduct{b,x}+c-h(x),
\end{align*}
where $A$ is an $n\times n$ matrix, $b\in\R^n$, $c\in\R$ and $h:\R^n\rightarrow \R$ is a smooth convex function whose gradient is not Lipschitz continuous, e.g., $h(x):=\norm{Cx-d}^4$, where $C$ is an $m\times n$ matrix and $d\in\R^m.$ Indeed, we can find some $L>0$ such that the matrix $LI-A$ is positive-semidefinite, where $I$ is the $n\times n$ identity matrix. It follows from the second-order characterization of convex functions that
\begin{align*}
\dfrac{L}{2}\norm{x}^2-\Big(\dfrac{1}{2}\dotproduct{Ax,x}+\dotproduct{b,x}+c\Big)=\dfrac{1}{2}\big\la(LI-A)x,x\big\ra-\dotproduct{b,x}-c\;\text{ is convex}.
\end{align*}
Combining this with the convexity of $h$, we get the convexity of $\dfrac{L}{2}\norm{x}^2-f(x),$ which means that $f$ satisfies the $L-$descent property \eqref{descent condition}.

Even when $\nabla f$ is Lipschitz continuous with constant $L>0$, $f$ can satisfy the $\tilde{L}-$descent condition with $\tilde{L}<L$. E.g., consider the univariate function $f$ together with its gradient $\nabla f$ given by
\begin{align*}
f(x)=\begin{cases}
\frac{3}{4}x^2&\text{\rm if }|x|<\frac{2}{3},\\
-\frac{3}{2}x^2+3x-1&\text{\rm if }\frac{2}{3}\le x\le 1,\\
-\frac{3}{2}x^2-3x-1&\text{\rm if }-1 \le x\le -\frac{2}{3},\\
|x|-\frac{x^2}{2}&\text{\rm if }|x|>1
\end{cases}\quad\text{ and }\quad \nabla f(x)=\begin{cases}
\frac{3}{2}x&\text{\rm if }|x|<\frac{2}{3},\\
-3x+3&\text{\rm if }\frac{2}{3}\le x\le 1,\\
3x-3&\text{\rm if }-1 \le x\le -\frac{2}{3},\\
1-x&\text{\rm if }x>1,\\
-1-x&\text{\rm if }x<1.
\end{cases}
\end{align*}
The latter representation implies that $L=3$ is the smallest constant for the Lipschitz continuity of $\nabla f$. Meanwhile, we see on Figure~\ref{fig:h2} that $f$ satisfies the $\tilde{L}$-descent property with $\tilde{L}=3/2.$
\begin{figure}[H]
\centering
\includegraphics[scale=0.3]{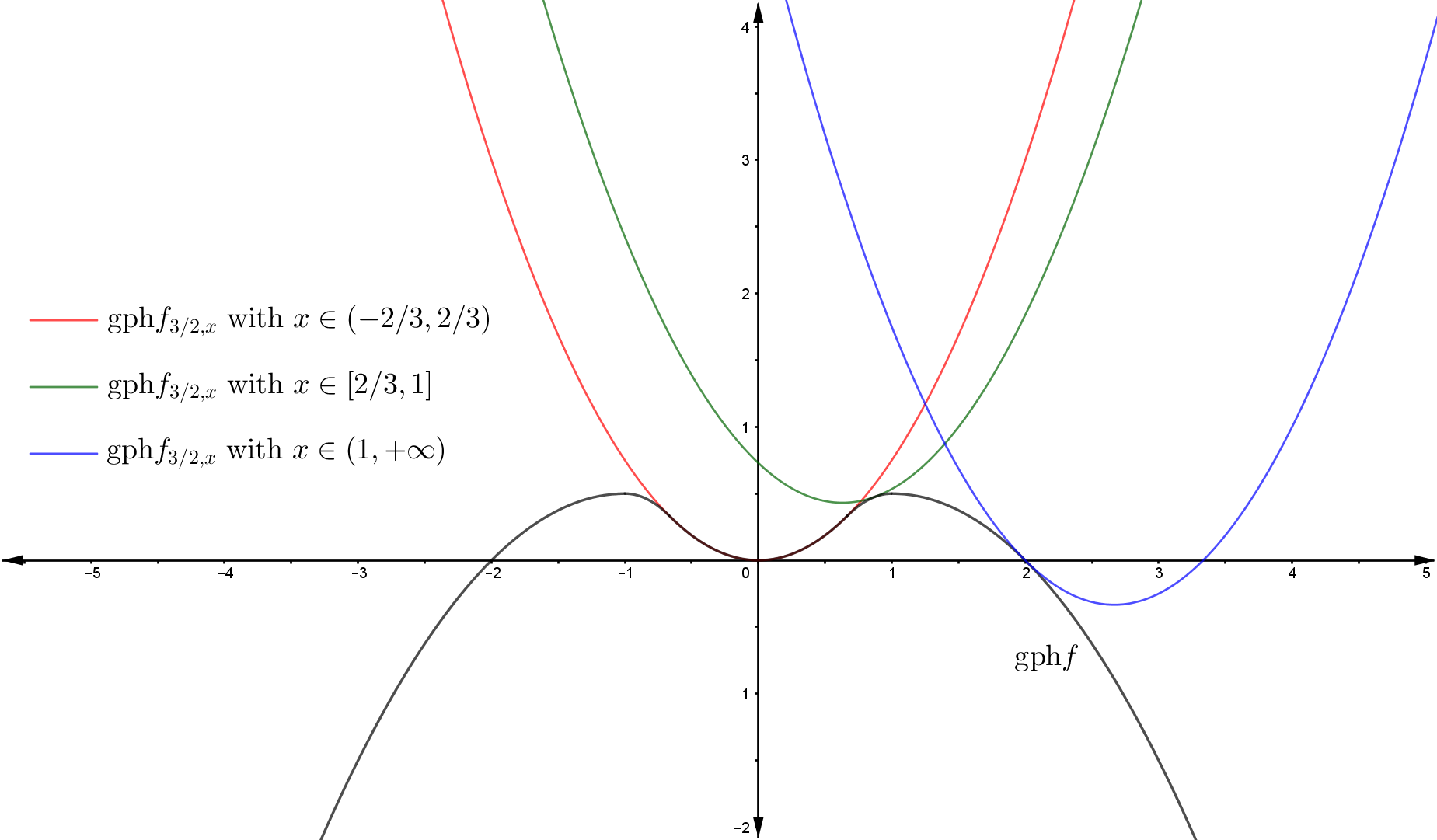}
\caption{An illustration for $f$ and $f_{3/2,x}$.}
\label{fig:h2}
\end{figure}
Next we recall by following \cite[Section~1.2]{bertsekasbook} some basic stepsize selections for the iterative procedure (\ref{intro: stepsize + descent direction}). The stepsize sequence $\set{t_k}$ satisfies the \textit{Armijo rule} if there exist a scalar $\beta$ and a reduction factor $\gamma\in(0,1)$ such that for all $k\in\N$ we have the representation
\begin{align}\label{armijo}
t_k=\max\big\{t\;\big|\;f(x^k+td^k)\le f(x^k)+\beta t\big\la\nabla f(x^k),d^k\big\ra,\;t=1,\gamma,\gamma^2,\ldots\big\}.
\end{align}
This stepsize selection ensures the nonincreasing property of the entire sequence $\set{f(x^k)}$. However, Armijo stepsizes may be small and thus require a large number of stepsize reducing steps in order to make just small changes of the iterative sequence.

To significantly simplify the iterative sequence design, it is possible to consider a {\em constant stepsize}, i.e., $t_k:=\alpha$ for all $k\in\N$. For this rule, the nonincreasing property of $\set{f(x^k)}$ is not ensured in general but holds under the $L$-descent condition (\ref{descent condition}) whenever $\alpha$ is chosen to be sufficiently small with respect to $1/L$, see, e.g., \cite{bertsekasbook,nocedalbook}. However, even when the $L$-descent condition is satisfied for $f$ while an approximate value of $L$ is unknown, using constant stepsizes becomes inefficient. In such a case, it is possible to use the {\em diminishing stepsize} selection, i.e.,
\begin{align}\label{diminishing step size def}
t_k\downarrow 0\;\;\mbox{ as }\;k\to\infty\;\text{ and }\;\sum_{k=1}^\infty t_k=\infty.
\end{align}
Drawbacks of the latter selection are the eventual {\em slow convergence} due to its small stepsizes and the absence of the descent property for the iterative sequence $\set{f(x^k)}$. 

\medskip
Now we formulate a general type of directions that plays a crucial role in our subsequent analysis of various linesearch methods.

\begin{Definition}\rm\label{defi: approximate gradient}
Let $\set{x^k}$ be a sequence in $\R^n$. The direction sequence $\set{d^k}$ is called \textit{gradient associated} with $\set{x^k}$ if we have the implication
\begin{align}\label{condition for stationarity}
\text{ whenever }\;d^k\overset{J}{\rightarrow}0\;\text{ for some infinite set }\;J\subset\N,\;\text{ it holds that }\;\nabla f(x^k)\overset{J}{\rightarrow}0.
\end{align}
\end{Definition}\vspace*{-0.02in}

It can be easily checked that if either
\begin{align}\label{dk grad 0}
\lim_{k\rightarrow\infty}\norm{d^k-\nabla f(x^k)}=0,
\end{align}
or there exists some constant $c>0$ such that
\begin{align}\label{d grad}
\norm{\nabla f(x^k)}\le c\norm{d^k}\;\text{ for all sufficiently large }\;k\in\N,
\end{align}
then $\set{d^k}$ is \textit{gradient associated} with $\set{x^k}$. Many methods such as the gradient descent, the generalized damped Newton method \cite{kmpt21,kmpt21.2}, and the methods appeared in \cite[Proposition~1.2.3]{bertsekasbook} satisfy (\ref{d grad}), while (\ref{dk grad 0}) can be considered as a standard condition for inexact gradient directions. It should be also mentioned that the notion of gradient associated directions is different from the notion of gradient related directions proposed by Bertsekas in \cite{bertsekasbook}.

\medskip
Finally in this section, we discuss two versions of the fundamental KL property playing plays a crucial role in the results on global convergence and convergence rates established in what follows. The first version, which is mainly used in the paper, is due to Absil et al. \cite[Theorem~3.4]{absil05}.

\begin{Definition}\rm \label{KL ine}\rm
Let $f:\R^n\rightarrow\R$ be a differentiable function. We say that $f$ satisfies the \textit{KL property} at $\bar{x}\in\R^n$ if there exist a number $\eta>0$, a neighborhood $U$ of $\bar{x}$, and a nondecreasing function $\psi:(0,\eta)\rightarrow(0,\infty)$ such that the function $1/\psi$ is integrable over $(0,\eta)$ and we have
\begin{align}\label{kl absil}
 \norm{\nabla f(x)}\ge\psi\big(f(x)-f(\bar{x})\big)
\end{align}
for all $x\in U$ with $f(\bar x)<f(x)<f(\bar x)+\eta$.
\end{Definition}

\begin{Remark}\rm \label{KL-attouch} \rm
For smooth functions $f$, the KL property from Definition~\ref{KL ine} is weaker than the KL property of $f$ at $\ox$ introduced by Attouch et al. in \cite{attouch10} meaning that there exist a number $\eta>0$, a neighborhood $U$ of $\bar{x}$, and a continuous concave function $\varphi:[0,\eta]\rightarrow[0,\infty)$ such that 
\begin{itemize}
\item[\bf(i)] $\varphi(0)=0$;
\item[\bf(ii)]$\varphi$ is $\mathcal{C}^1$-smooth on $(0,\eta)$;
\item[\bf(iii)] $\varphi'>0$ on $(0,\eta)$;
\item[\bf(iv)] for all $x\in U$ with $f(\bar{x})<f(x)<f(\bar{x})+\eta$, we have 
\begin{align}\label{KL 2}
\varphi'\big(f(x)-f(\bar{x})\big)\norm{\nabla f(x)}\ge 1.
\end{align}
\end{itemize}

Supposing that $f$ satisfies the KL property at $\ox$ in the sense of Attouch et al., we now show that it satisfies the KL property from Definition~\ref{KL ine}. Indeed, take $\eta,\ U$, and the function $\varphi$ satisfying conditions (i)-(iv) in Remark~\ref{KL-attouch}. By (iii), the function $\psi:(0,\eta)\rightarrow[0,\infty)$ given by
\begin{align*}
\psi(x):=1/\varphi'(x)\;\text{ for all }\;x\in(0,\eta)
\end{align*}
is well-defined. Since $\varphi$ is concave, $\varphi'$ is nonincreasing. Therefore, $1/\psi=\varphi'$ is a nondecreasing positive function that is integrable on $(0,\eta)$. Moreover, by $1/\psi=\varphi'$ we get that (\ref{KL 2}) and (\ref{kl absil}) are equivalent. Thus $f$ satisfies the KL property at $\bar x$ in the sense of Definition~\ref{KL ine}.

It has been realized that KL property in the sense of Attouch et al., and hence the one from Definition~\ref{KL ine}, is satisfied in broad settings. In particular, it holds at every {\em nonstationary point} of $f$; see \cite[Lemma~2.1~and~Remark~3.2(b)]{attouch10}. Furthermore, it is proved at the original paper by \L ojasiewicz \cite{lojasiewicz65} that any analytic function $f:\R^n\rightarrow\R$ satisfies the KL property at every point $\ox$ with $\varphi(t)~=~Mt^{1-q}$ for some $q\in [0,1)$. Typical smooth functions satisfying this property are {\em semialgebraic} functions and also those from the more general class of functions {\em definable in o-minimal structures}; see \cite{bolte06,kurdyka}. For other examples of functions satisfying the KL property, we refer the reader to \cite{attouch10,lewis09} and the bibliographies therein.
\end{Remark}\vspace*{-0.15in}

\section{Abstract Convergence Results for Linesearch Methods}\label{sec auxiliary}

In this section, we establish properties for a general class of linesearch methods of type (\ref{intro: stepsize + descent direction}), which provide major tools for convergence analysis of IRG methods in Section~\ref{sec 5}. One of the most important results desired for linesearch methods is as follows:
\begin{align}\label{stationary statement}
\text{ every accumulation point of }\;\{x^k\}\;\text{ is a stationary point of }\;f.
\end{align}
By the continuity of the gradient mapping, the desired property \eqref{stationary statement} automatically holds if for each accumulation point $\bar x$ of $\set{x^k}$ we can find an infinite set $J\subset\N$ such that $x^k\overset{J}{\rightarrow}\bar x$ and $\nabla f(x^k)\overset{J}{\rightarrow}0$. If the exact information on the gradient is unknown while $\set{d^k}$ is gradient associated with $\set{x^k}$, i.e., \eqref{condition for stationarity} is satisfied, then property (\ref{stationary point lemma}) is satisfied when
\begin{align}\label{abstract tk dk}
x^k\overset{J}{\rightarrow}\bar x\;\text{ and }\;d^k\overset{J}{\rightarrow
}0\;\text{ for some infinite set }\;J\subset\N.
\end{align}
We are going to show that (\ref{abstract tk dk}) holds whenever $0$ is an accumulation point of $\set{d^k}$ and
\begin{align}\label{series is finite}
\sum_{k=1}^\infty\norm{x^{k+1}-x^k}\cdot\norm{d^k}<\infty.
\end{align}
The following lemma is new and is crucial in the convergence analysis for many linesearch methods.

\begin{Lemma}\label{stationary point lemma}
Let $\set{x^k}$ and $\set{d^k}$ be sequences satisfying 
\eqref{series is finite}. If $\bar x$ is an accumulation point of $\set{x^k}$ and if $0$ is an accumulation point of $\set{d^k}$, then there exists an infinite set $J\subset\N$ such that
\begin{align}\label{relation dk xk 3}
x^k\overset{J}{\rightarrow}\bar x\;\mbox{ and }\;d^k\overset{J}{\rightarrow}0.
\end{align}
\end{Lemma}
\begin{proof}
If $x^k\rightarrow\bar x$ as $k\to\infty$, the conclusion obviously holds, so suppose that $x^k\not\rightarrow \bar x$. It suffices to show that for any $\delta>0$ sufficiently small and any $N\in\N$ there is a number $k_N\ge N$ such that
\begin{align*}
\|x^{k_N}-\bar x\|<\delta\;\text{ and }\;\|d^{k_N}\|<\delta.
\end{align*}
Fix such $\delta>0$ and $N\in\N$. Since $\delta$ is sufficiently small and $x^k\not\rightarrow\bar x$, suppose that the set
\begin{align*}
A_1:=\big\{k\ge N\;\big|\;\|x^k-\bar x\|\ge\delta\big\}\;\text{ is infinite}.
\end{align*}
As $\bar x$ is an accumulation point of $\set{x^k}$ and $0$ is an accumulation point of $\set{d^k}$, we get that
\begin{align*}
A_2:&=\big\{k\ge N\;\big|\;\|x^k-\bar x\|<\delta\big\}\;\text{ is infinite},\\
A_3:&=\big\{k\ge N\;\big|\;\|x^k-\bar x\|<\delta/2\big\}\;\text{ is infinite},\\
 A_4:&=\big\{k\ge N\;\big|\;\|d^k\|<\delta\big\}\;\text{ is infinite}.
\end{align*}
It suffices to verify that $A_2\cap A_4\ne\emptyset.$ Suppose on the contrary that $A_2\cap A_4=\emptyset$, i.e.,
\begin{align*}
\|d^k\|\ge\delta\;\text{ for all }\;k\in A_2.
\end{align*}
By (\ref{series is finite}), we have the estimates
\begin{align*}
\infty> \sum_{k=1}^\infty\norm{x^{k+1}-x^k}\cdot\norm{d^k}&\ge \sum_{k\in A_2} \norm{x^{k+1}-x^k}\cdot\norm{d^k}\\
&\ge \delta\sum_{k\in A_2}\norm{x^{k+1}-x^k},
\end{align*}
which ensure the series convergence
\begin{align}\label{sum bounded 1}
\sum_{k\in A_2}\norm{x^{k+1}-x^k}<\infty.
\end{align}
Taking any number $K\in A_3$, we also have $K\in A_2.$ Since $A_1$ is infinite and  $A_1,A_2$ form a partition of the set $\set{N,N+1,\ldots}$ including $K$, there exists a number $\widehat{K}\in A_1$ with $\widehat{K}>K$ such that $K,K+1,\ldots,\widehat{K}-1\in A_2$. Then we have the estimates
\begin{align*}
\norm{x^{\widehat{K}}-x^K}\ge\norm{x^{\widehat{K}}-\bar x}-\norm{x^K-\bar x}\ge\delta-\delta/2=\delta/2.
\end{align*}
Using the triangle inequality and (\ref{sum bounded 1}) gives us
\begin{align*}
\delta/2\le\norm{x^{\widehat{K}}-x^K}\le \sum_{i=K}^{\widehat{K}-1}\norm{x^{i+1}-x^i}\le\sum_{i\ge K,\;i\in A_2}\norm{x^{i+1}-x^i}\xrightarrow[K\in A_3]{K\rightarrow\infty}0,
\end{align*}
which brings us to a contradiction that completes the proof of the lemma.
\end{proof}

\begin{Remark}\rm Our proof of Lemma \ref{stationary point lemma} is inspired by the proofs of Kiwiel in \cite[Theorem~3.3(iii)]{kiwiel07}, \cite[Theorem~3.8(iii)]{kiwiel10} and of Burke and Lin in \cite[Theorem~5]{burke21}. If $0$ is not an accumulation point of $\{d^k\}$, then
\eqref{series is finite} implies that $\{x^k\}$ is convergent. Indeed, the negation of the statement that $0$ is an accumulation point of $\{d^k\}$ yields the existence of $\tau>0$ and $K\in\N$ such that  
$$
\|d^k\|\ge\;\tau\text{ for all }\;k\ge K.
$$
It follows from \eqref{series is finite} that    
$$
\tau\sum_{k=K}^\infty \norm{x^{k+1}-x^k}\le
\sum_{k=K}^\infty \norm{x^{k+1}-x^k}\cdot\norm{d^k}<\infty,
$$
which implies that $\{x^k\}$ is a Cauchy sequence,  and thus it converges.
\end{Remark}

Next we recall classical results that describe important properties of the set of accumulation points generated by a sequence satisfying a limit condition introduced by Ostrowski; see (\ref{8.3.8}) below. The result from Lemma~\ref{lemma: ostrowski}(i) was first established in \cite[Theorem~28.1]{ostrowski}, while assertion (ii) of this lemma appeared in \cite{fachi95.1}. These results are also stated as Theorem~8.3.9 and Theorem~8.3.10 in the book of Facchinei and Pang \cite{pangbook2}. More recent usage of the of the Ostrowski condition in the framework of linesearch methods can be found in \cite[Lemma~5(iii)]{bolte14}. 
\begin{Lemma}\label{lemma: ostrowski}
Let $\set{x^k}$ be a sequence satisfying the Ostrowski condition
\begin{align}\label{8.3.8}
\lim_{k\rightarrow\infty}\|x^{k+1}-x^k\|=0.
\end{align}
Then the following assertions hold:
\begin{itemize}
\item[\bf(i)] If $\set{x^k}$ is bounded, then the set of accumulation points of $\set{x^k}$ is nonempty, compact, and connected in $\R^n$.

\item[\bf(ii)] If $\set{x^k}$ has an isolated accumulation point, then this sequence converges to it.
\end{itemize}
\end{Lemma}

Now we are ready to establish the main result of this section revealing major convergence properties of a general class of linesearch methods.

\begin{Theorem}\label{stationary point theorem}
Let $\set{x^k}$ be a sequence generated by a linesearch method \eqref{intro: stepsize + descent direction} such that:
\begin{itemize}
\item[\bf(a)] $\set{d^k}$ is gradient associated with $\set{x^k}$;
\item[\bf(b)] $0$ is an accumulation point of $\set{d^k};$
\item[\bf(c)] $\displaystyle{\sum_{k=1}^\infty}t_k\|d^k\|^2<\infty.$
\end{itemize}
Then every accumulation point of $\set{x^k}$ is a stationary point of $f$. Moreover, if $\set{t_k}$ is bounded from above, then the following assertions hold:
\begin{itemize}
\item[\bf(i)] If $\set{x^k}$ is bounded, then the set of accumulation points of $\set{x^k}$ is nonempty, compact, and connected.
\item[\bf(ii)] If $\set{x^k}$ has an isolated accumulation point, then this sequence converges to it.
\end{itemize}
\end{Theorem}
\begin{proof}
Let $\bar x$ be an accumulation point of $\set{x^k}$.  Note that (c) is equivalent to \eqref{series is finite} under the linesearch relationship $x^{k+1}=x^k+t_kd^k$ in (\ref{intro: stepsize + descent direction}). Applying Lemma~\ref{stationary point lemma} with taking into account (b) and (c), we can find an infinite set $J\subset\N$ such that $x^k\overset{J}{\rightarrow }\bar x$ and $d^k\overset{J}{\rightarrow}0$. Then (a) implies that $\nabla f(x^k)\overset{J}{\rightarrow}0$. Employing the continuity of $\nabla f$, we have 
\begin{align*}
\nabla f(\bar x)=\lim_{k\overset{J}{\rightarrow}\infty}\nabla f(x^k)=0,
\end{align*}
which tells us that $\bar x$ is a stationary point of $f$. Suppose now that $\set{t_k}$ is bounded from above by some $\tau>0$. Using (c), we immediately get
\begin{align*}
\sum_{k=1}^\infty\|x^{k+1}-x^k\|^2=\sum_{k=1}^\infty t_k^2\|d^k\|^2\le\tau\sum_{k=1}^\infty t_k\|d^k\|^2<\infty.
\end{align*}
This leads us to
$\norm{x^{k+1}-x^k}\rightarrow 0$ and verifies assertions (i) and (ii) by applying Lemma~\ref{lemma: ostrowski}.
\end{proof}

Theorem~\ref{stationary point theorem} also allows us to ensure the stationarity of accumulation points generated by linesearch methods (\ref{intro: stepsize + descent direction}) applied to functions satisfying the $L$-descent condition (\ref{descent condition}), where the stepsize is either constant or diminishing, and where the direction is \textit{gradient associated} while satisfying the following \textit{sufficient descent} condition
\begin{align}\label{(b)}
 \la\nabla f(x^k),d^k\ra\le -\kappa\|d^k\|^2\;\text{ for all }\;k\in\N
\end{align}
with some constant $\kappa>0$. Note that condition (\ref{(b)}) is different from the gradient associated condition from Definition~\ref{defi: approximate gradient}, since from (\ref{(b)}) we only have that $\nabla f(x^k)\overset{J}{\rightarrow}0$ yields $d^k\overset{J}{\rightarrow}0$ but the reverse implication may not hold. In addition to the gradient descent and generalized Newton methods discussed above, there exist many other linesearch methods using direction (\ref{(b)}); e.g., the boosted difference of convex functions algorithm as in \cite[Proposition~4]{aragon18}), the inexact Levenberg-Marquardt method as in \cite[Algorithm~3.1]{dan02}), and the GS method for nonsmooth functions with non-normalized direction given in \cite[Section~4.1]{kiwiel07}.\vspace*{0.05in} 

We have the following effective consequence of Theorem~\ref{stationary point theorem}.

\begin{Corollary}\label{convergence diminishing}
Let $\set{x^k}$ be a sequence generated by a linesearch method \eqref{intro: stepsize + descent direction}. Suppose that $\inf f(x^k)>-\infty$, and that we have the conditions:
\begin{itemize}
\item[\bf(a)] $f$ satisfies the $L$-descent condition \eqref{descent condition} for some $L>0$;
\item[\bf(b)] the sequence $\set{d^k}$ is gradient associated with $\set{x^k}$ and satisfies \eqref{(b)} for some $\kappa>0;$
\item[\bf(c)] the sequence $\set{t_k}$ is not summable, i.e.,
\begin{align}\label{def dimi}
\displaystyle{\sum_{k=1}^\infty} t_k=\infty,
\end{align}
and there are numbers $\delta>0$ and  $N\in\N$ such that
\begin{align}\label{defi cons}
 t_k\le \dfrac{2\kappa-\delta}{L}\;\text{ for all }\;k\ge N.
\end{align}
\end{itemize}
Then every accumulation point of $\set{x^k}$ is stationary for $f$, and the following assertions hold:
\begin{itemize}
\item[\bf(i)] If the sequence $\set{x^k}$ is bounded, then the collection of accumulation points of $\set{x^k}$ is nonempty, compact, and connected.
\item[\bf(ii)] If $\set{x^k}$ has an isolated accumulation point, then the entire sequence $\set{x^k}$ converges to it.
\end{itemize}
Moreover, if $\set{t_k}$ is bounded away from $0,$ then $\nabla f(x^k)\rightarrow0$ as $k\to\infty$. 
\end{Corollary}
\begin{proof}
It follows  from (\ref{defi cons}) in condition (c) that 
\begin{align*}
\kappa-\dfrac{Lt_k}{2}\ge \dfrac{\delta}{2}\;\text{ for all }\;k \ge N.
\end{align*}
Since $f$ satisfies $L$-descent condition (\ref{descent condition}), we deduce from (\ref{(b)}) and the latter inequality that 
\begin{align}\label{descent of dimi and constant}
f(x^{k+1})&\le f(x^k)+t_k\dotproduct{\nabla f(x^k),d^k}+\dfrac{Lt_k^2}{2}\norm{d^k}^2\nonumber\\
&\le f(x^k)-t_k\kappa\norm{d^k}^2 +\dfrac{Lt_k^2}{2}\norm{d^k}^2\nonumber\\
&= f(x^k)-t_k\norm{d^k}^2\Big(\kappa-\dfrac{Lt_k}{2}\Big)\nonumber\\
&\le f(x^k)-\dfrac{\delta}{2}t_k\norm{d^k}^2\;\text{ for all }\;k\ge N.
\end{align} 
Then summing up the relationships $\dfrac{\delta}{2}t_k\norm{d^k}^2\le f(x^k)-f(x^{k+1})$ over $k=N,N+1,\ldots$ and using the assumption $\inf f(x^k)>-\infty$ give us 
\begin{align}\label{tk dk impo}
\sum_{k=N}^\infty t_k\norm{d^k}^2<\infty.
\end{align}
Now we show that $0$ is an accumulation point of $\set{d^k}$. Suppose on the contrary that there exist a positive number $u$ and a natural number $K\geq N$ such that \begin{align*}
\norm{d^k}\ge u\;\text{ for all }\;k\ge K.
\end{align*}
Using this together with (\ref{tk dk impo}) imply that  $\sum_{k=K}^\infty t_k<\infty$, which contradicts (\ref{def dimi}). Therefore, $0$ is an accumulation point of $\set{d^k}$. 
Combining the latter with (b), (\ref{tk dk impo}), and (\ref{defi cons}) allows us to confirm that all the assumptions of Theorem~\ref{stationary point theorem} are satisfied. Thus every accumulation point of $\{x_k\}$ is a stationary point of $f$, and both assertions in (i) and (ii) hold.

If finally $\set{t_k}$ is bounded away from $0,$ it follows from (\ref{tk dk impo}) that $d^k\rightarrow0$ as $k\to\infty$. Since the sequence $\set{d^k}$ is gradient associated with $\set{x^k}$ by (b), we get $\nabla f(x^k)\rightarrow 0$.
\end{proof}

\begin{Remark}\rm 
Let us compare Corollary~\ref{convergence diminishing} with two well-known results by Bertsekas \cite{bertsekasbook} about the convergence of linesearch methods. First we compare the  result of Corollary~\ref{convergence diminishing} in the case where $\set{t_k}$ is bounded away from $0$ with Bertsekas's result for the constant stepsize.

 \begin{Proposition}{\rm \cite[Proposition~1.2.2]{bertsekasbook}}
 Let $\set{x^k}$ be a sequence generated by a linesearch method $x^{k+1}=x^k+t_kd^k$, where the sequence $\set{d^k}$ is gradient related. Assume that 
\begin{align}\label{Lipschitz}
\norm{\nabla f(x)-\nabla f(y)}\le L\norm{x-y}\;\text{ for all }\;x,y\in\R^n
\end{align}
with some constant $L>0$, and that for all $k\in\N$ we have $d^k\ne 0$ and 
\begin{align}
\varepsilon\le t_k\le (2-\varepsilon)\bar t_k,
\end{align}
where $\varepsilon$ is a fixed positive scalar, and where
\begin{align*}
\bar t_k:=\dfrac{\abs{\dotproduct{\nabla f(x^k),d^k}}}{L\norm{d^k}^2}.
\end{align*}
Then every limit point of $\set{x^k}$ is a stationary point of $f$.
 \end{Proposition}

 Our comparison with this result are as follows:
\begin{enumerate}
\item Observe first that the result of \cite[Proposition~1.2.2]{bertsekasbook} assumes the Lipschitz continuity of $\nabla f$, while the proof works under the $L$-descent condition (\ref{descent condition}). Thus there is no difference between the class of functions considered in \cite[Proposition~1.2.2]{bertsekasbook} and Corollary~\ref{convergence diminishing}.

\item Regarding the stepsize selection, our choice is much simpler than the one used in \cite[Proposition~1.2.2]{bertsekasbook}. Observe that in the direction selection, we use $\set{d^k}$ as \textit{gradient associated} with $\set{x^k}$ and satisfying the sufficient descent condition, while \cite[Proposition~1.2.2]{bertsekasbook} used the gradient related direction. It is commented in the note after \cite[Proposition~1.2.2]{bertsekasbook} that if the direction is chosen better, i.e., there are two numbers $c_1,c_2>0$ satisfying 
\begin{align}\label{1.28}
c_1\norm{\nabla f(x^k)}^2\le -\dotproduct{\nabla f(x^k),d^k},\quad \norm{d^k}^2\le c_2\norm{\nabla f(x^k)}^2,\quad k\in\N, 
\end{align}
then the stepsize selection can be reduced to the one chosen above. It can be seen that (\ref{1.28}) yields the estimates
\begin{align*}
\dotproduct{\nabla f(x^k),d^k}\le- c_1\norm{\nabla f(x^k)}^2 \le- \dfrac{c_1}{c_2}\norm{d^k}^2\;\text{ for all }\;k\in\N,
\end{align*}
and for any infinite set $J$ we have $d^k\overset{J}{\rightarrow}0$ if and only if $\nabla f(x^k)\overset{J}{\rightarrow}0$. Therefore, condition  (b) in Corollary~\ref{convergence diminishing} and the gradient related condition share the common direction (\ref{1.28}). 

\item Comparing the conclusions of Corollary~\ref{convergence diminishing} and of \cite[Proposition~1.2.2]{bertsekasbook}, observe that our result establishes the convergence to $0$ of the entire sequence $\set{\nabla f(x^k)}$, which is stronger than the stationarity of accumulation points obtained by \cite[Proposition~1.2.2]{bertsekasbook}. Moreover, our results about the set of accumulation points and convergence to an isolated accumulation point by using the Ostrowski condition are not given in \cite[Proposition~1.2.2]{bertsekasbook}.
\end{enumerate}

Next we compare the usage of the diminishing stepsize in Corollary~\ref{convergence diminishing} with the the corresponding result by Bertsekas \cite{bertsekasbook}.

\begin{Proposition}{\rm\cite[Proposition~1.2.3]{bertsekasbook}}
Let $\set{x^k}$ be a sequence generated by a linesearch method $x^{k+1}=x^k+t_kd^k$. Assume that $\nabla f$ is $L$-Lipschitzian, and that there exist positive scalars $c_1,c_2$ such that \eqref{1.28} holds for all $k\in\N$. Suppose also that 
\begin{align}\label{dimsum}
t_k\rightarrow 0\;\mbox{ and }\;\sum_{k=0}^\infty t_k=\infty.
\end{align}
Then either $f(x^k)\rightarrow-\infty$, or $\set{f(x^k)}$ converges to a finite value with $\nabla f(x^k)\rightarrow 0$ as $k\to\infty$. Furthermore, every accumulation point of $\set{x^k}$ is a stationary point of $f$.
\end{Proposition}

Here are our comparison with the latter proposition. 

\begin{enumerate}
\item The $L$-Lipschitz continuity of $\nabla f$ yields $L$-descent condition while the converse implication fails; see the commentaries in Section~\ref{sec prelim}. The convergence of $\set{\nabla f(x^k)}$ to $0$ is the main result of \cite[Proposition~1.2.3]{bertsekasbook}, and technically it cannot be derived without the $L$-Lipschitz continuity of $\nabla f$ or its small modifications. Relaxing the Lipschitz continuity of the gradient to the descent condition may allow us to work with broader classes of functions by replacing the usual Euclidean distance by the {\em Bregman distance}; see, e.g., \cite{baus17,bregman,nesterov18}.

\item The diminishing stepsize  (\ref{dimsum}) satisfies both conditions in Corollary~\ref{convergence diminishing}(c), while the converse implication may not hold.

\item As already mentioned above, condition (\ref{1.28}) yields (b) in Corollary~\ref{convergence diminishing}.

\item The conclusions of Corollary~\ref{convergence diminishing} provides the convergence to an isolated stationary point, while \cite[Proposition~1.2.3]{bertsekasbook} does not establish this.
\end{enumerate}

To summarize, if considering the stationarity of accumulation points only, Corollary~\ref{convergence diminishing} covers \cite[Proposition~1.2.3]{bertsekasbook} with a larger class of objective functions as well as descent directions. On the other hand, the convergence of $\set{\nabla f(x^k)}$ to $0$ is not obtained in Corollary~\ref{convergence diminishing} since the $L$-Lipschitz continuity is relaxed therein.
\end{Remark}

Next we present the convergence result for linesearch methods under the fulfillment of the \textit{KL property}, which is taken from \cite[Theorem~3.4]{absil05}.

\begin{Proposition}\label{general convergence under KL}
Let $f:\R^n\rightarrow\R$ be a $\mathcal{C}^1$-smooth function, and let the sequence of iterations $\set{x^k}\subset\R^n$ satisfy the following conditions:
\begin{itemize}
\item[\bf(H1)] {\rm(primary descent condition)}. There exists $\sigma>0$ such that for sufficiently large $k\in\N$ we have 
\begin{align*}
 f(x^k)-f(x^{k+1})\ge\sigma\norm{\nabla f(x^k)}\cdot\norm{x^{k+1}-x^k}.
\end{align*}
\item[\bf(H2)] {\rm(complementary descent condition)}. For sufficiently large $k\in\N$, we have
\begin{align*}
 \big[f(x^{k+1})=f(x^k)\big]\Longrightarrow [x^{k+1}=x^k].
\end{align*}
\end{itemize}
If $\bar x$ is an accumulation point of $\set{x^k}$  and $f$ satisfies the KL property at $\bar x$, then $x^k\rightarrow\bar x$ as $k\to\infty$.
\end{Proposition}

Note that \cite[Theorem~3.2]{absil05} provides a detailed proof for Proposition~\ref{general convergence under KL} when $f$ satisfies the KL property at $\bar x$ with $\psi(t)=Mt^q$ for some $M>0$ and $q\in [0,1)$. The key ingredient of the proof is the existence of some $k_1\in\mathbb{N}$ such that
\begin{align}\label{kl induction}
 \norm{x^{k+1}-x^k}&\le  \dfrac{1}{\sigma M(1-q )}\sbrac{(f(x^k)-f(\bar x))^{1-q }-(f(x^{k+1})-f(\bar x))^{1-q }}\text{ for all }k\ge k_1.
\end{align}
Using this, we can deduce that
\begin{align}\label{sum series cauchy}
\sum_{k=1}^\infty \norm{x^{k+1}-x^k}<\infty,
\end{align}
and so the sequence $\set{x^k}$ is convergent.

\medskip
Next we address {\em rates of convergence} of linesearch methods of type \eqref{intro: stepsize + descent direction}. Finding convergence rates for the gradient descent method in nonconvex settings is an important issue that has received an increasing attention recently; see, e.g., \cite{karimi16,rotaru22} and the references therein. Our results in Theorem~\ref{general rate} establish the rate of convergence of general linesearch methods under the KL property and the boundedness of stepsizes away from $0$. The proof below employs the technique developed in the proof of \cite[Theorem~2]{attouch09} for the proximal point method. This technique is also used in the proofs of convergence rates for the difference of convex algorithm \cite{an09} and its boosted version  in \cite[Theorem~1]{aragon18}. The next technical lemma taken from \cite[Lemma~1]{aragon18} is needed in what follows.

\begin{Lemma}\label{lemma KL}
Let $\set{s_k}$ be a nonnegative sequence in $\R$, and let $\alpha,\beta$ be some positive constants. Suppose that $s_k\rightarrow 0$ as $k\to\infty$ and that this sequence satisfies the condition
\begin{align*}
s_k^\alpha\le\beta(s_k-s_{k-1})\;\text{ for all }\;k\in\N\;\text{ sufficiently large}.
\end{align*}
The following assertions hold:
\begin{enumerate}
\item[\bf (i)] If $\alpha\in(0,1]$, then $s_{k+1}\le\brac{1-\dfrac{1}{\beta}}s_k$ for all $k\in\N$ sufficiently large, i.e., $\set{s_k}$ converges linearly to $0$ with rank $1-\dfrac{1}{\beta}.$
\item[\bf (ii)] If $\alpha>1$, then there exists a number $\varrho>0$ such that 
\begin{align*}
 s_k\le\varrho k^{-\frac{1}{\alpha-1}}\;\text{ for all }\;k\in\N\;\text{ sufficiently large}.
\end{align*}
\end{enumerate}
\end{Lemma}

Here is the aforementioned theorem about the convergence rates.

\begin{Theorem}\label{general rate}
Let the sequences $\set{x^k}\subset\R^n$, $\set{t_k}\subset\R_+$ and the numbers $\beta>0,\;c>0$ be such that $x^{k+1}\ne x^k$ for all $k\in\N$, and that we have
\begin{align}\label{two conditions}
f(x^k)-f(x^{k+1})\ge \dfrac{\beta}{t_k}\norm{x^{k+1}-x^k}^2\;\text{ and }\;\norm{\nabla f(x^k)}\le\dfrac{c}{t_k}\norm{x^{k+1}-x^k}
\end{align}
for sufficiently large $k\in\N$. Suppose that the sequence $\set{t_k}$ is bounded away from $0$, that $\bar x$ is an accumulation point of $\set{x^k}$, and that $f$ satisfies the KL property at $\bar x$ with $\psi(t)=Mt^{q}$ for some $M>0$ and $q\in(0,1)$. The following convergence rates are guaranteed:
\begin{itemize}\vspace*{-0.05in}
\item[\bf(i)] If $q\in(0,1/2]$, then the sequence $\set{x^k}$ converges linearly to $\bar x$.\vspace*{-0.05in}
\item[\bf(ii)]If $q\in(1/2,1)$, then there exists a positive constant $\varrho$ such that 
\begin{align*}
\norm{x^k-\bar x}\le \varrho k^{-\frac{1-q}{2q-1}}\;\text{ for sufficiently large }\;k\in\N.
\end{align*}
\end{itemize}
\end{Theorem}\vspace*{-0.15in}
\begin{proof}
The conditions imposed in (\ref{two conditions}) imply that
\begin{align*}
f(x^k)-f(x^{k+1})\ge \dfrac{\beta}{c}\norm{\nabla f(x^k)}\cdot\norm{x^{k+1}-x^k}\;\text{ for sufficiently large }\;k\in\N,
\end{align*}
which ensures that assumption (H1) in Proposition~\ref{general convergence under KL} holds with $\sigma:=\dfrac{\beta}{c}$. By
$t_k>0$ for all $k\in\N$, the first condition in (\ref{two conditions}) yields also assumption (H2) of this proposition. Let $k_1\in\N$ be such that \eqref{kl induction} is satisfied for all $k\ge k_1$, which implies in turn inequality \eqref{sum series cauchy} and thus the convergence of $\set{x^k}$ to $\ox$. Moreover, the first condition in (\ref{two conditions}) together with $x^{k+1}\ne x^k$ as $k\in\N$ implies that 
\begin{align}\label{f(xk) converges}
f(x^k)\downarrow f(\bar x):=\bar f\;\text{ and }\;f(x^k)>\bar f\;\text{ for all }\;k\in\N.
\end{align}
It follows from (\ref{sum series cauchy}) and the triangle inequality that
\begin{align*}
\norm{x^i-\bar x}\le\sum_{k=i}^\infty\norm{x^{k+1}-x^k}=:s_i<\infty\;\text{ whenever }\;i\in\N.
\end{align*}
Therefore, the convergence rate of $x^i$ to $\bar x$ can be deduced from the convergence rate of $s_i$ to $0$. Let $k_1$ be defined in (\ref{kl induction}) and pick $i\ge k_1$. Summing up (\ref{kl induction}) from $i$ to any $p>i$ gives us
\begin{align*}
\sum_{k=i}^{p}\norm{x^{k+1}-x^k}\le\dfrac{1}{\sigma M(1-q)}\big(f(x^i)-\bar f\big)^{1-q}.
\end{align*}
Letting $p\rightarrow\infty$, we get the estimate
\begin{align}\label{sikm}
s_i\le\dfrac{1}{\sigma M(1-q )}\big(f(x^i)-\bar f\big)^{1-q}\;\text{ for all }\;i\ge k_1.
\end{align}
The KL property of $f$ at $\bar x$ with $\psi(t)=Mt^{q }$ allows us to find a positive number $\eta$ and a neighborhood $U$ of $\bar x$ ensuring the condition
\begin{align}\label{KL M gradient}
\norm{\nabla f(x)}\ge M\big(f(x)-\bar f\big)^q\;\mbox{ whenever }\;x\in U,\;\bar f<f(x)<\bar f+\eta.
\end{align} 
Using (\ref{KL M gradient}), $x^k\rightarrow\bar x$, and (\ref{f(xk) converges}) gives us a natural number $k_2>k_1$ such that
\begin{align}\label{M(1-theta)}
\norm{\nabla f(x^i)}\ge M\big(f(x^i)-\bar f\big)^q\;\text{ for all }\;i\ge k_2.
\end{align}
Let $\bar{t}>0$ be a lower bound of $\set{t_k}$. Combining this with (\ref{M(1-theta)}) and the second inequality in (\ref{two conditions}), we have the estimates
\begin{align}
\big(f(x^i)-\bar f\big)^q &\le\dfrac{1}{M}\norm{\nabla f(x^i)}\le\dfrac{c}{Mt_i}\norm{x^{i+1}-x^i}\nonumber\\
&\le\dfrac{c}{M\bar{t}}\norm{x^{i+1}-x^i}\;\text{ for all }\;i\ge k_2.\label{3.220}
\end{align}
Involving further (\ref{sikm}) and (\ref{3.220}) ensures that
\begin{align*}
s_i^{\frac{q}{1-q }}&\le\brac{\dfrac{1}{\sigma M(1-q)}}^{\frac{q }{1-q}}\big(f(x^i)-\bar f\big)^q\\
&\le\brac{\dfrac{1}{\sigma M(1-q)}}^{\frac{q }{1-q}}\dfrac{c}{M\bar{t}}\norm{x^{i+1}-x^i}.
\end{align*}
Define $\alpha:=\dfrac{q }{1-q }$ and $\beta:=\brac{\dfrac{1}{\sigma M(1-q )}}^{\frac{q }{1-q }}\dfrac{c}{M\bar{t}}$, and then rewrite the latter estimate as
\begin{align*}
s_i^\alpha\le\beta(s_i-s_{i-1})\;\text{ for all }\;i\ge k_2.
\end{align*}
Applying finally Lemma~\ref{lemma KL}, we verify the convergence rates in (i), (ii), and (iii). 
\end{proof}

By employing the iterative procedure $x^{k+1}=x^k+t_kd^k,$ the conditions in \eqref{two conditions} can be rewritten as the following estimates: 
\begin{align*}
f(x^k)-f(x^{k+1})\ge \beta t_k\norm{d^k}^2\;\text{ and }\;\norm{\nabla f(x^k)}\le c \norm{d^k}.
\end{align*}

Note  that, in addition to the IRG methods as shown in Section~\ref{sec 5},  these conditions are also satisfied for the generalized damped Newton method as proved in \cite[Theorem~4.2]{kmpt21}.\vspace*{-0.15in}

\section{General Scheme for Inexact Reduced Gradient Methods}\label{sec Stationarity of accumulation points}\vspace*{-0.05in}

In this section, we design a general framework for our novel IRG methods and establish their basic properties prior to constructing particular methods of this type with various stepsize selections. As mentioned in Section~\ref{sec:intro}, the motivation to design  the IRG methods comes from the gradient sampling (GS) method by Burke et al. \cite{blo02} and further developments in Kiwiel \cite[Section~4.1]{kiwiel07} for unconstrained problems of nonsmooth optimization.

Although the focus of this paper is on smooth optimization problems, we first recall the {\em perturbed} notions of generalized differentiation for nonsmooth functions and their smooth specifications employed in \cite{blo02,kiwiel07} as well as different ones used in our algorithms. In what follows, the notation $\B(\ox,\ve)$ stands for the closed ball centered at $\ox\in\R^n$ with radius $\ve>0$, by ${\rm co}\,\O$ we understand the convex hull of a set $\O\subset\R^n$, and the symbol $x\st{\O}{\to}\ox$ indicates that $x\to\ox$ with $x\in\O$.\vspace*{0.05in} 

Let $f\colon\R^n\to\R$ be a real-valued function locally Lipschitzian around some point $\ox\in\R^n$. The generalized differential notion employed in \cite{blo02,kiwiel07}
is introduced by Goldstein \cite{gol77} under the name of the $\ve$-{\em generalized gradient} of $f$ at $\ox$. It is defined by
\begin{equation}\label{egg}
\Bar\partial_\varepsilon f(\bar x):={\rm co}\,\Bar\partial f(\mathbb{B}\big(\bar x,\varepsilon)\big),\quad\ve>0,
\end{equation}
via the convex hull of the (Clarke) {\em generalized gradient}
\begin{equation}\label{cgg}
\Bar\partial f(\bar x):={\rm co}\Big\{\lim_{x\overset{\Omega}{\rightarrow}\bar x}\nabla f(x)\Big\},
\end{equation}
where $\Omega$ denotes the set of full measure on which $f$ is differentiable; see \cite{clarkebook,rockafellarbook} for more details. If $f$ is ${\cal C}^1$-smooth around $\ox$, then $\Bar\partial f(\ox)$ in \eqref{cgg} reduces to the gradient $\{\nabla f(\ox)\}$, and hence the $\ve$-generalized gradient \eqref{egg} is represented by
\begin{equation}\label{egg-smooth}
\Bar\partial_\ve f(\bar x)={\rm co}\,\nabla f\big(\mathbb{B}(\bar x,\varepsilon)\big),\quad
\ve>0. 
\end{equation}

Prior to introducing our IRG methods, let us briefly describe the GS scheme based on the $\ve$-generalized gradient in the smooth case \eqref{egg-smooth}. Recall that the 
{\em projection} ${\rm Proj}(x,\O)$ of $x\in\R^n$ onto a closed and convex set $\O\subset\R^n$ is a singleton that is characterized by 
\begin{align}\label{proj}
u={\rm Proj}(x,\O)\;\text{ if and only if $u\in\O$ and  }\;\dotproduct{y-u,x-u}\le 0\;\text{ for all }\;y\in\O.
\end{align}
Given a ${\cal C}^1$-smooth function $f\colon\R^n\to\R$, an initial point $x^1$, initial radii $\varepsilon_1,r_1$, and a natural number $m\ge n+1$, the GS method approximately calculates the $\varepsilon_k$-generalized gradient of $f$ at the $k^{\rm th}$ iteration $x^k$ by taking the convex hull of the gradients
\begin{align}\label{Gk}
\partial_{\varepsilon_k}f(x^k)\approx{\rm co}\set{\nabla f(x^k),\nabla f(x^{k,1}),\nabla f(x^{k,2}),\ldots, \nabla f(x^{k,m})}=:\mathcal{G}_k,
\end{align}
where the points $x^{k,i}$ with $i=1,\ldots,m$ are chosen uniformly in $\mathbb{B}(x^k,\varepsilon_k)$. If $\mathcal{G}_k\cap\mathbb{B}(0,r_k)\ne\emp$, then both $\varepsilon_k$ and $r_k$ are reduced in the sense that $\ve_{k+1}=\mu\ve_k$ and $r_{k+1}=\theta r_k$ for some $\mu,\th\in(0,1)$, 
while the reference point $x^k$ stays the same. Otherwise, the new iterative point $x^{k+1}$ is updated by the procedure $x^{k+1}=x^k+t_kd^k$, where $-d^k={\rm Proj}(0,\mathcal{G}_k)$, and where $t_k$ is given by the backtracking linesearch. This procedure is illustrated on Figure~\ref{fig:GS} for the case where $n=2$ and $m=3$. 
\begin{figure}[H]
\centering
\includegraphics[scale=0.3]{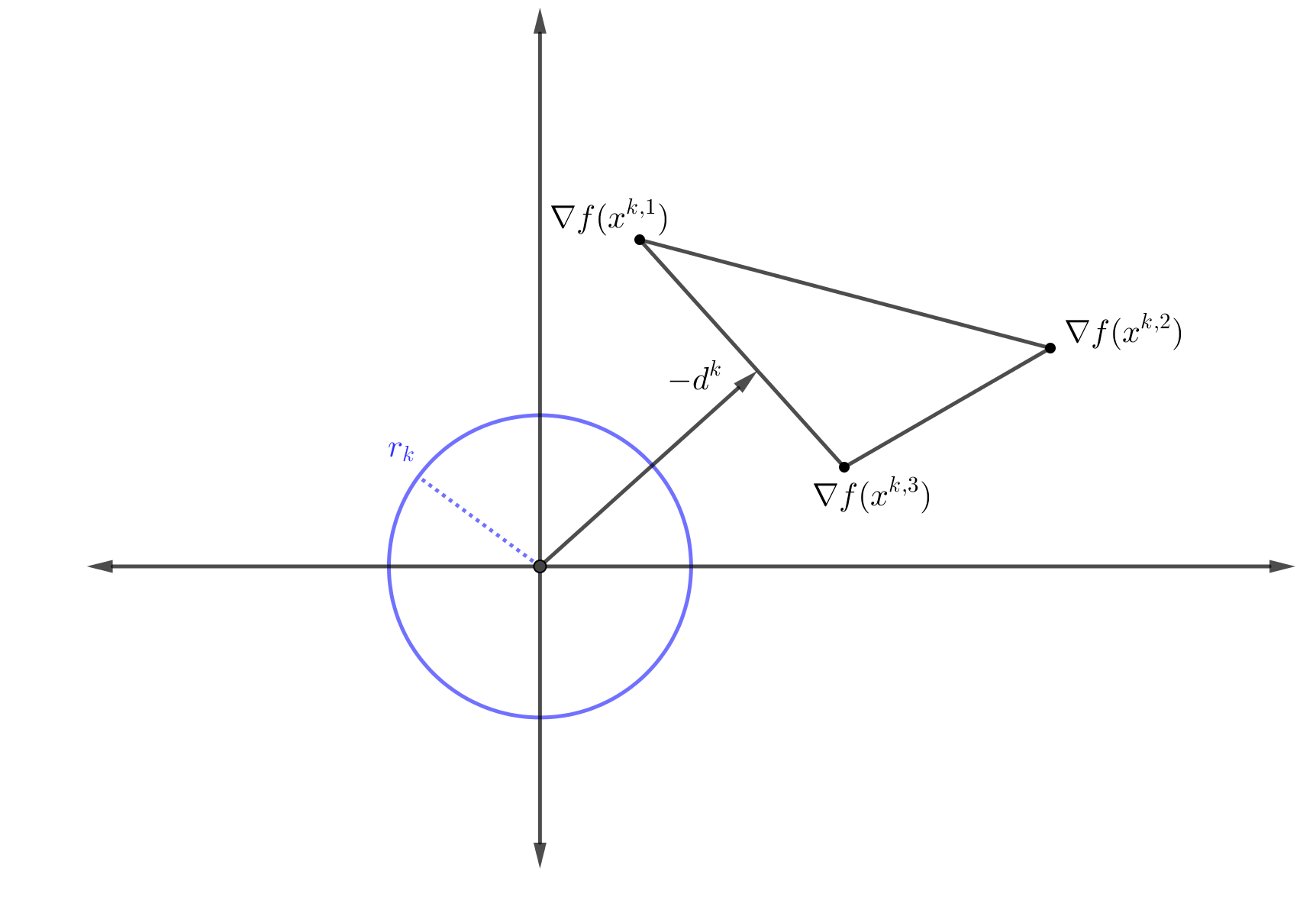}
\caption{Illustration for the GS method when $x^k$ is updated}
\label{fig:GS}
\end{figure}

In contrast to \cite{blo02,kiwiel07}, in our IRG methods proposed in this paper, we use another subdifferential construction for locally Lipschitzian functions $f$ at $\ox$ defined as follows
\begin{equation}\label{ef}
\widehat{\partial}_\varepsilon f(\bar x):=\set{v\in\R^n\;\Big|\;\liminf_{x\rightarrow\bar x}\dfrac{f(x)-f(\bar x)-\la v,x-\bar x\ra}{\norm{x-\bar x}}\ge-\varepsilon},\quad\ve\ge 0.
\end{equation}
This construction is known in variational analysis and optimization as the (Fr\'echet) $\ve$-{\em subdifferential} of $f$ at $\ox$. When $\ve=0$, \eqref{ef} reduces to the {\em regular subdifferential} $\Hat\partial f(\ox)$; see, e.g., \cite{mordukhovich06,mordukhovich18,rockafellarbook} and the references therein. For $\ve>0$, the $\ve$-subdifferential goes back to the early work by Kruger and Mordukhovich, where it was used via a limiting procedure to adequately extend the notion of the (Mordukhovich) {\em limiting subdifferential} $\partial f(\ox)$ from $\R^n$ to {\em Banach} spaces; see \cite{mordukhovich06,mordukhovich18} with more discussions and references. In this paper, we show that \eqref{ef} for $\ve>0$ is useful to design and justify {\em numerical algorithms} for {\em smooth} problems in {\em finite dimensions}.\vspace*{0.05in} 

The next proposition gives us the expression of \eqref{ef} for differentiable functions.

\begin{Proposition}
Let $f:\R^n\rightarrow\R$ be differentiable at $\bar x\in\R^n$, and let $\varepsilon>0$. Then we have 
\begin{equation}\label{ef-smooth}
\widehat{\partial}_\varepsilon f(\bar x)=\B\big(\nabla f(\bar x),\varepsilon\big).
\end{equation}
\end{Proposition}
\begin{proof}
Since $f$ is differentiable at $\bar x$, we have 
\begin{align}\label{diff frechet}
\lim_{x\rightarrow\bar x}\dfrac{f(x)-f(\bar x)-\dotproduct{\nabla f(\bar x),x-\bar x}}{\norm{x-\bar x}}=0.
\end{align}
Take any $v\in\mathbb{B}(\nabla f(\bar x),\varepsilon)$ and write $v=\nabla f(\bar x)+h$ with $\norm{h}\le \varepsilon$. It follows from (\ref{diff frechet}) that \begin{align*}
\liminf_{x\rightarrow\bar x}\dfrac{f(x)-f(\bar x)-\dotproduct{\nabla f(\bar x)+h,x-\bar x}}{\norm{x-\bar x}}&=\liminf_{x\rightarrow\bar x}\dfrac{-\dotproduct{h,x-\bar x}}{\norm{x-\bar x}}\\
&\ge \liminf_{x\rightarrow\bar x}\dfrac{-\norm{h}\cdot\norm{x-\bar x}}{\norm{x-\bar x}}\\ &=-\norm{h}\ge -\varepsilon.
\end{align*}
This means that $v\in\widehat{\partial}_\varepsilon f(\bar x)$, and so $\widehat{\partial}_\varepsilon f(\bar x)\supset\mathbb{B}(\nabla f(\bar x),\varepsilon)$. To verify the converse implication, pick any $v\in\widehat{\partial}_\varepsilon f(\bar x)$ meaning that 
\begin{align*}
\liminf_{x\rightarrow\bar x}\dfrac{f(x)-f(\bar x)-\dotproduct{v,x-\bar x}}{\norm{x-\bar x}}\ge-\varepsilon
\end{align*}
and deduce from the latter and \eqref{diff frechet} that 
\begin{align}\label{17}
\liminf_{x\rightarrow\bar x}\dfrac{\dotproduct{\nabla f(\bar x)-v,x-\bar x}}{\norm{x-\bar x}}\ge -\varepsilon.
\end{align}
Then for any $e\in\mathbb{B}(0,1)$ we get from \eqref{17} and $\bar x\pm \frac{1}{k}e\rightarrow\bar x$ as $k\rightarrow\infty$ that
\begin{align*}
\lim_{k\rightarrow\infty}\dfrac{\dotproduct{\nabla f(\bar x)-v,\frac{1}{k}e}}{\norm{\frac{1}{k}e}}\ge -\varepsilon\quad\text{ and }\quad \lim_{k\rightarrow\infty}\dfrac{\dotproduct{\nabla f(\bar x)-v,-\frac{1}{k}e}}{\norm{-\frac{1}{k}e}}\ge -\varepsilon,
\end{align*}
which implies in turn that $\abs{\dotproduct{\nabla f(\bar x)-v,e}}\le \varepsilon$. Since this holds for all $e\in\mathbb{B}(0,1)$, it follows that $\norm{\nabla f(\bar x)-v}\le\varepsilon$, i.e., $v\in \mathbb{B}(\nabla f(\bar x),\varepsilon)$, which verifies the claimed representation \eqref{ef-smooth}.
\end{proof}

The following table summarizes the main differences between the $\ve$-generalized gradient in \eqref{egg-smooth} and $\ve$-subdifferential in \eqref{ef-smooth} for smooth functions. 

\begin{table}[H]
\centering
\begin{tabular}{ |l|l|l| } 
\hline
&\multicolumn{1}{c|}{$\Bar\partial_\varepsilon f(\bar x)={\rm co}\,\nabla f\big(\mathbb{B}(\bar x,\varepsilon)\big)$} & \multicolumn{1}{c|}{$\widehat{\partial}_\varepsilon f(\bar x)=\mathbb{B}\big(\nabla f(\bar x),\varepsilon\big)$} \\
\hline
How to calculate? &$\approx{\rm co}\set{\nabla f(x^k)\;\big|\;x^k\in\mathbb{B}(\bar x,\varepsilon),\;k\in\{1,\ldots,m\},\;m\ge n+1}$&has explicit form \\
\hline
How to find projection?&$\approx$ projection onto the convex hull above& has explicit form
 \\
 \hline 
\end{tabular}
\caption{Comparison between the $\varepsilon$-generalized gradient and the $\ve$-subdifferential for smooth functions}
\label{compare}
\end{table}

We can see from Table~\ref{compare} that the numerical implementation of algorithms involving $\Hat\partial_\ve f$ is {\em much easier} in comparison with the GS method using $\Bar\partial_\ve f$ for smooth functions. One more advantage of algorithms implementing the $\ve$-subdifferential \eqref{ef} in comparison with the usage of $\ve$-generalized gradient \eqref{egg} is that, for nonsmooth functions, algorithms based on \eqref{egg} lead us to {\em Clarke stationary} points satisfying $0\in\Bar\partial f(\ox)$ in terms of the generalized gradient \eqref{cgg}. On the other hand, the usage of \eqref{ef} brings us to the more subtle {\em Mordukhovich stationary} $0\in\partial f(\ox)$ due to the representation
\begin{equation*}
\partial f(\ox)=\big\{v\in\R^n\;\big|\;\exists\,\ve_k\dn 0,\;x_k\to\ox,\;v_k\to v\;\mbox{ with }\;v_k\in\Hat\partial_{\ve_k}f(x_k)\big\}
\end{equation*}
of the limiting subdifferential, which works even in Banach spaces with the appropriate convergence of $v_k\to v$; cf.\ \cite[Theorem~1.28]{mordukhovich18} and \cite[Theorem~1.89]{mordukhovich06}. Since we have
\begin{equation*}
\Bar\partial f(\ox)={\rm co}\,\partial f(\ox)   
\end{equation*}
when $f\colon\R^n\to\R$ is locally Lipschitzian around $\ox$, the condition $0\in\partial f(\ox)$ is more selective than that of $0\in\Bar\partial f(\ox)$ to find stationary points and minimizers of Lipschitzian functions and will be investigated in our future development of IRG algorithms for nonsmooth optimization problems.\vspace*{0.1in}

Now we are ready to formulate our {\em general algorithmic framework} (the {\em Master Algorithm}) for IRG methods with the usage of the $\ve$-subdifferential representation \eqref{ef-smooth} for smooth functions while without considering yet particular stepsize selections.

 \begin{Algorithm}[\bf general framework for IRG methods]\rm \label{al gd inexact}
\quad
\setcounter{Step}{-1}
\begin{Step}[initialization]\rm Select an initial point $x^1\in \R^n$, initial radii $\varepsilon_1,r_1> 0$, radius reduction factors $\mu,\theta\in(0,1)$, a sequence of manually controlled errors $\set{\rho_k}\subset \R_+.$
\end{Step}
\begin{Step}[inexact gradient and stopping criterion]\rm \label{Step 1}
Choose $g^k$ such that 
\begin{align}\label{calculate approx grad}
\norm{g^k-\nabla f(x^k)}\le \min\set{\varepsilon_k,\rho_k}.
\end{align}
\end{Step}
\begin{Step}[radius update]\rm \label{Step 2}
If $\norm{g^k}\le r_{k}+\varepsilon_k$, then set $r_{k+1}:=\mu r_k,\;\varepsilon_{k+1}:=\theta\varepsilon_k$, $d^k:=0$, and go to~Step~3. Otherwise, set $r_{k+1}:=r_k$, $\varepsilon_{k+1}:=\varepsilon_k$, and 
\begin{align}\label{choose dk}
d^k:=-\dfrac{\norm{g^k}-\varepsilon_k}{\norm{g^k}}g^k.
\end{align}
\end{Step}
\begin{Step}[stepsize]\rm\label{Step: 3}
Choose $t_k>0$ by a specific rule.
\end{Step}
\begin{Step}[iteration update]\rm\label{step 4}
Set $x^{k+1}:=x^{k}+t_kd^k$.
\end{Step}
\begin{Step}\rm
Increase $k$ by $1$ and go back to Step $1.$
\end{Step}
\end{Algorithm}

Let us make some comments on the constructions of Algorithm~\ref{al gd inexact}. The first remark concerns the novelty in the choice of errors and directions.

\begin{Remark}\label{remark intu}\rm We have the following observations:\vspace*{-0.05in}
\begin{enumerate}[\bf(i)]

\item The selection of $\set{\rho_k}$ satisfying $\rho_k\rightarrow0$ as $k\rightarrow\infty$ plays a vital role in the convergence analysis of the IRG method with backtracking stepsize as shown later. In the remaining cases, since the presence of $\rho_k$ is not necessary, we specify $\rho_k:=\varepsilon_k$ for all $k\in\N$, which leads us to a simplified form of the condition \eqref{calculate approx grad} as follows:
 \begin{align}\label{IRG condition constant}
\norm{g^k-\nabla f(x^k)}\le \varepsilon_k.
\end{align}\vspace*{-0.2in}

\item From \eqref{calculate approx grad}, the radius $\varepsilon_k$ can be considered as an automatically controlled error for the calculation of $\nabla f(x^k)$, which does not need to decrease after each iteration. This is different than  the choice $\varepsilon_k=ck^{-p}$ for $p\ge 1,c>0$ frequently used in the well-known methods \cite{bertsekasbook,nesterov14,gilmore95}. Moreover, Steps~\ref{Step 1} and \ref{Step 2} also show that $\norm{\nabla f(x^k)}\le r_k+2\varepsilon_k$ when $\varepsilon_k$ is reduced. Therefore, we can conclude intuitively that $\norm{\nabla f(x^k)}$ is decreasing when $\varepsilon_k$ is decreasing.\vspace*{-0.05in}

\item In the {\em exact case} when $g^k=\nabla f(x^k)$ for all $k\in\N$, we label our methods as  the {\em reduced gradient} (RG) ones, which are different from the standard gradient descent method. Indeed, it follows from Step~\ref{Step 2} that $d^k$ is either $0$, or is given by 
\begin{align}\label{reduced gradient}
d^k=-\Big(\dfrac{\|\nabla f(x^k)\|-\varepsilon_k}{\|\nabla f(x^k)\|}\Big)\nabla f(x^k).
\end{align}
Therefore, the vector $d^k$ in (\ref{reduced gradient}) has the same direction as $-\nabla f(x^k)$ but its length is $\norm{\nabla f(x^k)}$ reduced by $\varepsilon_k$ for each $k\in\N$.
\end{enumerate}
\end{Remark}\medskip
Further we discuss and illustrate behavior of Algorithm~\ref{al gd inexact} at the major steps of iterations.

\begin{Remark}\rm\quad\label{remark 1} Notice first that:
\begin{itemize}\vspace*{-0.05in}

\item[\bf(i)] If $d^k\ne 0$, it follows from (\ref{choose dk}) and the definition of projections that $d^k=-{\rm Proj}(0,\mathbb{B}(g^k,\varepsilon_k))$.\vspace*{-0.05in}

\item[\bf(ii)] An illustration for Algorithm~\ref{al gd inexact} can be seen in Figure~\ref{fig:1}.

Let $g^1$ be an approximate gradient of $\nabla f(x^1)$ at the $1$st iteration. Then Figure~\ref{fig:1} shows that the two balls $\mathbb{B}(g^1,\varepsilon_1)$ and $\mathbb{B}(0,r_1)$ do not intersect. This means by Step~\ref{Step 2} that $r_2=r_1$, $\varepsilon_2=\varepsilon_1$, and $d^1=-{\rm Proj}(0,\mathbb{B}(g^1,\varepsilon_1))$. Then we have a new point $x^2=x^1+t_1d^1$ after choosing the stepsize $t_1>0$ as in Step~\ref{Step: 3} and Step~\ref{step 4}.

At the $2$nd iteration, it can be seen in Figure~\ref{fig:1} that the two balls $\mathbb{B}(g^2,\varepsilon_2)$ and $\mathbb{B}(0,r_2)$ intersect each other. Thus by Step~\ref{Step 2} of Algorithm~\ref{al gd inexact}, the radii $r_2,\varepsilon_2$ are reduced to $r_3=\mu r_2$ and $\varepsilon_3=\theta\varepsilon_2$, while the direction $d^2$ is zero. The latter means that the iterative point $x^2$ stays the same, i.e., $x^3=x^2$ from Step~\ref{step 4}.

At the $3$rd iteration, although $\nabla f(x^3)=\nabla f(x^2)$, we still need to recalculate an approximate gradient $g^3$ with a new error $\min\set{\varepsilon_3,\rho_3}$. In this iteration, the two balls $\mathbb{B}(g^3,\varepsilon_3)$ and $\mathbb{B}(0,r_3)$ do not intersect, and hence the procedure is similar to that at the first iteration.
\begin{center}
\begin{figure}[H]
\centering
\includegraphics[scale=0.4]{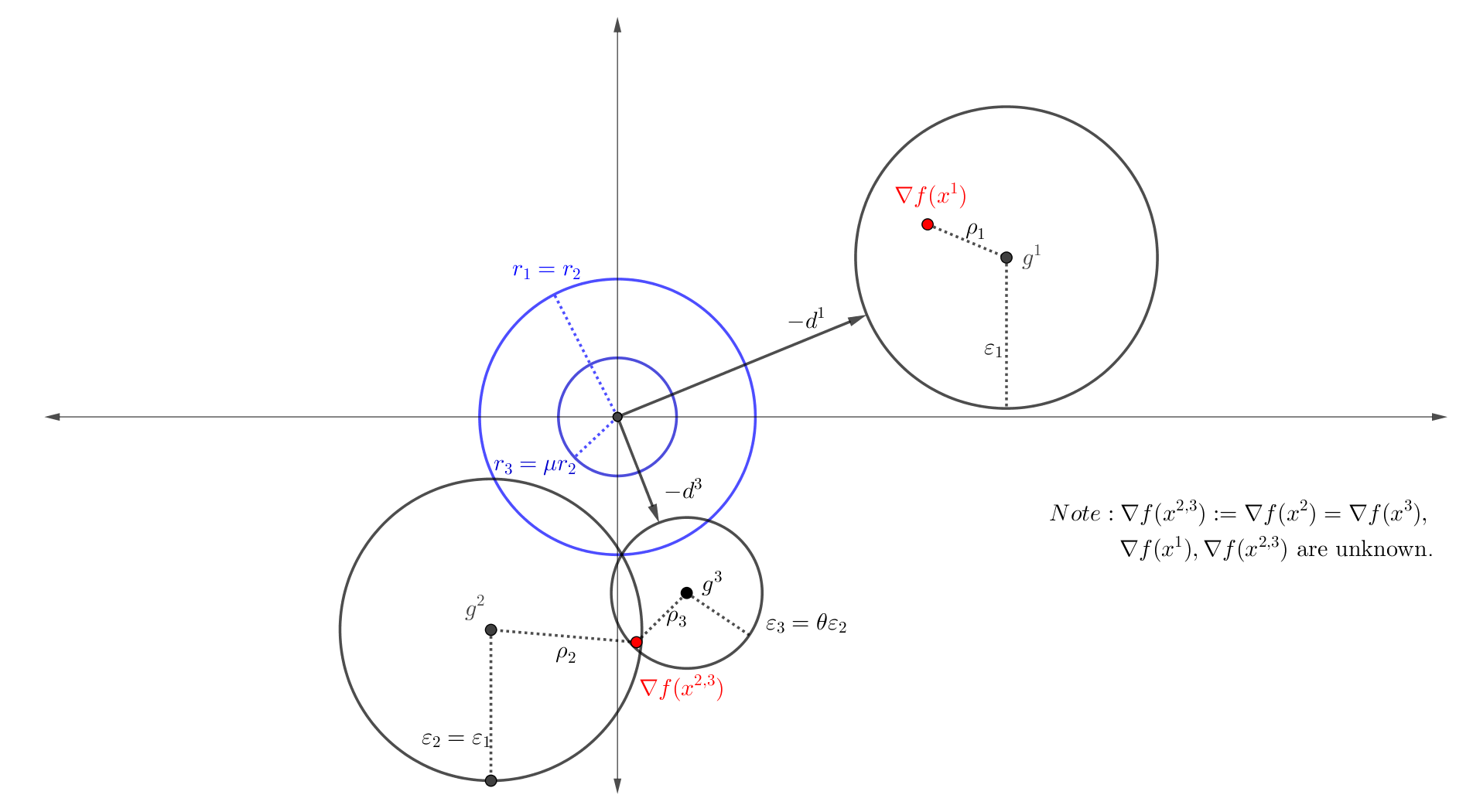}
\caption{An illustration for IRG methods}
\label{fig:1}
\end{figure}
\end{center}
\vspace*{-0.4in}
\item[\bf(iii)] For each $k\in\N$, we have from Step~\ref{Step 2} and Step~\ref{Step: 3} the equivalences
\begin{align}\label{alter condi for null}
x^{k+1}=x^k\Longleftrightarrow d^k=0\Longleftrightarrow r_{k+1}=\mu r_k\Longleftrightarrow\varepsilon_{k+1}=\theta\varepsilon_k\Longleftrightarrow\norm{g^k}\le r_k+\varepsilon_k.
\end{align}
\end{itemize}
\end{Remark}

The next proposition verifies the {\em decent property} of Algorithm~\ref{al gd inexact}.

\begin{Proposition}\label{prop sufi holds}
In Algorithm~{\rm\ref{al gd inexact}}, $\set{d^k}$ satisfies the sufficient descent condition with constant~$1,$ i.e., 
\begin{align}\label{sufi des holds}
\dotproduct{\nabla f(x^k),d^k}\le -\norm{d^k}^2\;\text{ for all }\;k\in\N.
\end{align}
\end{Proposition}

\begin{proof}
Note that (\ref{sufi des holds}) automatically holds if $d^k=0$. Supposing that $d^k\ne 0$ and using the construction of $d^k$ in Step~\ref{Step 2}, we have the expression
\begin{align*}
-d^k={\rm Proj}\big(0,\mathbb{B}(g^k,\varepsilon_k)\big).
\end{align*}
It follows from (\ref{calculate approx grad}) that $\norm{g^k-\nabla f(x^k)}\le \varepsilon_k$, which means that $\nabla f(x^k)\in \mathbb{B}(g^k,\varepsilon_k)$. Invoking the projection description for convex sets \eqref{proj} yields
\begin{align*}
\dotproduct{0+d^k,\nabla f(x^k)+d^k}\le 0,
\end{align*}
which is in turn equivalent to
\begin{align*}
\dotproduct{\nabla f(x^k),d^k}\le-\norm{d^k}^2
\end{align*}
and thus verifies the claim in \eqref{sufi des holds}.
\end{proof}

Now we introduce the notion of null iterations and establish some properties of such iterations related to the IRG methods.

\begin{Definition}\rm\label{nul iter}
The $k^{\rm th}$ iteration of Algorithm~{\rm\ref{al gd inexact}} is called a {\em null iteration} if $x^{k+1}=x^k$. The set of all null iterations is denoted by
$$
\mathcal{N}:=\big\{k\in\N\;\big|\;x^{k+1}=x^k\big\}.
$$
\end{Definition}\vspace*{0.03in}

The next proposition collects important properties of null iterations. 

\begin{Proposition}\label{null properties}
Let $\set{x^k},\set{g^k},\set{d^k},\set{\varepsilon_k}$, and $\set{r_k}$ be sequences generated by Algorithm~{\rm\ref{al gd inexact}}. The following assertions hold:
\begin{itemize}
\item[\bf(i)] $k\in\mathcal{N}$ if and only if either one of the equivalent conditions in \eqref{alter condi for null} holds.\vspace*{-0.05in}

\item[\bf(ii)] $\varepsilon_k\downarrow 0$ if and only if $r_k\downarrow 0$, which is equivalent to the set $\mathcal{N}$ being infinite.

\item[\bf(iii)] If $\mathcal{N}$ is finite, then we have
\begin{align*}
 \norm{g^k}>r_N+\varepsilon_N\;\text{ and }\;\norm{d^k}>r_N
\end{align*}
for all $k\ge N$, where $N:=\max\mathcal{N}+1$.

\item[\bf(iv)] If $\N\setminus\mathcal{N}$ is finite, then $\nabla f(x^K)=0$ and $\set{x^k}_{k\ge K}$ is a constant sequence, where we denote $K:=\max\{\N\setminus\mathcal{N}\}+1$. Otherwise, $\bar x$ is an accumulation point of $\set{x^k}$ if and only if it is an accumulation point of $\set{x^k}_{k\in \N\setminus\mathcal{N}}$, and therefore $x^k\rightarrow\bar x$ if and only if $x^k\overset{\N\setminus\mathcal{N}}{\longrightarrow}\bar x$.
\end{itemize}
\end{Proposition}
\begin{proof} Assertions (i) and (ii) follow directly from Definition~\ref{nul iter}. To verify (iii), observe that for any natural number $k\ge N$, the $k^{\rm th}$ iteration is not a null one and then deduce from (i) that
\begin{align*}
\varepsilon_{k+1}=\varepsilon_k=\varepsilon_N,\quad r_{k+1}=r_k=r_N,\;\mbox{ and }\;\norm{g^k}>r_k+\varepsilon_k=r_N+\varepsilon_N.
\end{align*}
Together with Step~\ref{Step 2} in Algorithm~\ref{al gd inexact}, this ensures that   
\begin{align*}
\norm{d^k}=\norm{g^k}-\varepsilon_k>r_k=r_N.
\end{align*}
which readily justifies assertion (i).

The proof of (iv) is a bit more involved. Supposing that the set $\N\setminus\mathcal{N}$ is finite, we have that $k\in \mathcal{N}$ for all $k\ge K$. This means by (i) that
\begin{align*}
x^{k+1}=x^{k},\;\varepsilon_{k+1}=\theta\varepsilon_k,\;r_{k+1}=\mu r_k,\;\text{ and }\;\norm{g^k}\le r_k+\varepsilon_k\;\text{ whenever }\;k\ge K.
\end{align*}
This tells us that $x^k=x^{K}$ for all $k\ge K$, and that $\varepsilon_k\downarrow 0,\;r_k\downarrow 0$, and $g^k\rightarrow 0$ as $k\to\infty$. Taking the limit as $k\rightarrow \infty$ in $\norm{g^k-\nabla f(x^k)}\le \varepsilon_k$ gives us $\nabla f(x^k)\rightarrow0$, which yields $\nabla f(x^{K})=0$.

Supposing otherwise that the set $\N\setminus\mathcal{N}$ is infinite, we obviously get that every accumulation point of $\set{x^k}_{k\in \N\setminus \mathcal{N}}$ is an accumulation point of $\set{x^k}$. Conversely, taking any accumulation point $\bar x$ of $\set{ x^k}$, it suffices to show that 
\begin{align*}
\mbox{for any }\;\delta>0,\;N\in\N\;\mbox{ there exists }\;k_N\in\N\setminus\mathcal{N},\;k_N\ge N\;\mbox{ with }\;\norm{x^{k_N}-\bar x}<\delta.
\end{align*}
To verify this, fixing $\delta >0$ and $N\in\N$ and remembering that $\bar x$ is an accumulation point of $\set{x^k}$, we find $K\ge N$ such that $\norm{x^K-\bar x}<\delta$. If $K\in\N\setminus\mathcal{N}$, choose $k_N:=K$. Otherwise, using that $\N\setminus\mathcal{N}$ is infinite allows us to find $\widehat{K}\in\N\setminus\mathcal{N}$ for which $K,K+1,\ldots,\widehat{K}-1\in\mathcal{N}$. This ensures that 
\begin{align*}
{x^{\widehat{K}}}={x^{\widehat{K}-1}}=\ldots=x^{K+1}=x^K,
\end{align*}
and therefore, with $k_N:=\widehat{K}$, we get that $\norm{x^{k_N}-\bar x}<\delta$. Since $\delta$ was chosen arbitrarily, this clearly shows that $\bar x$ is an accumulation point of $\set{x^k}_{k\in\N\setminus\mathcal{N}}$ and thus completes the proof. 
\end{proof}

The last proposition here establishes relationships between convergence properties of the sequences $\set{g^k}$ and $\set{d^k}$ in Algorithm~\ref{al gd inexact}.

\begin{Proposition}\label{goes to 0} Let $\set{g^k},\set{d^k},\;\set{\varepsilon_k}$, and $\set{r_k}$ be sequences generated by Algorithm~{\rm\ref{al gd inexact}}. Then for any $k\in\N$ we have the estimates
\begin{align}\label{nx1}
\norm{d^k}\le\norm{g^k}\le\norm{d^k}+\varepsilon_k+r_k.
\end{align}
Consequently, the following assertions hold:
\begin{itemize}
\item[\bf(i)] $\varepsilon_k\downarrow0$ if and only if there is an infinite set $J\subset\N$ such that $g^k\overset{J}{\rightarrow}0$.

\item[\bf(ii)] For any infinite set $J\subset\N$, we have the equivalence
\begin{align*}
g^k\overset{J}{\rightarrow}0\Longleftrightarrow d^k\overset{J}{\rightarrow}0.
\end{align*}
\end{itemize}
\end{Proposition}

\begin{proof}
Fix any $k\in\N$. If $k\in\mathcal{N}$, then we get by Proposition~\ref{null properties}(i) that $d^k=0$ and $\norm{g^k}\le \varepsilon_k+r_k$. Otherwise, Step~\ref{Step 2} in Algorithm~\ref{al gd inexact} yields $\norm{d^k}=\norm{g^k}-\varepsilon_k\le \norm{g^k}$. In both cases, (\ref{nx1}) holds. 

To deduce (i) from (\ref{nx1}), suppose that $\varepsilon_k\downarrow 0$ as $k\to\infty$. By Proposition~\ref{null properties}(ii) we have that $r_k\downarrow 0$ and the set $\mathcal{N}$ is infinite. Then for any $\delta>0$ and $N\in\N$ there is $k\ge N$ with $k\in\mathcal{N}$ such that
\begin{align*}
\norm{g^{k}}\le r_k+\varepsilon_k<\delta,
\end{align*}
where the first inequality follows from (\ref{alter condi for null}). Thus we can construct an infinite set $J\subset\N$ such that $g^k\overset{J}{\longrightarrow}0$. If conversely the sequence $\set{\varepsilon_k}$ does not converge to $0$, then the set $\mathcal{N}$ is finite by Proposition~\ref{null properties}(ii). Using Proposition~\ref{null properties}(iii) confirms that $\set{g^k}$ is bounded away from $0$, which tells us that such an index set $J$ does not exist.

To verify now assertion (ii), observe that assuming $g^k\overset{J}{\rightarrow}0$, implies by the first inequality in (\ref{nx1}) that $d^k\overset{J}{\rightarrow}0$. Conversely, suppose that $d^k\overset{J}{\rightarrow}0$ and deduce from Proposition~\ref{null properties}(i) that the set $\mathcal{N}$ is infinite. Then it follows from Proposition~\ref{null properties}(ii) that  $\varepsilon_k\downarrow0$ and $r_k\downarrow0$ as $k\to\infty$. Using the second inequality in (\ref{nx1}), we arrive at $g^k\overset{J}{\rightarrow}0$ and thus complete the proof of the proposition.
\end{proof}

Finally in this section, we deduce from the obtained results the following desired property of the direction sequence $\set{d^k}$ in Algorithms~\ref{al gd inexact}. 

\begin{Corollary}\label{overall} The sequence $\set{d^k}$ in Algorithms~{\rm\ref{al gd inexact}} is gradient associated with $\set{x^k}$.
\end{Corollary}\vspace*{-0.2in}
\begin{proof} It follows from Proposition \ref{goes to 0} that the convergence $d^k\overset{J}{\rightarrow}0$ yields $g^k\overset{J}{\rightarrow}0$ and $\varepsilon_k\downarrow0.$ Thus we get  $\nabla f(x^k)\overset{J}{\rightarrow}0$ by taking into account $\norm{g^k-\nabla f(x^k)}\le \varepsilon_k$ from \eqref{calculate approx grad}. This shows therefore that the sequence $\set{d^k}$ is gradient associated with $\set{x^k}$.
\end{proof}\vspace*{-0.15in}

\section{Inexact Reduced Gradient Methods with Stepsize Selections}\label{sec 5}

In this section, we develop novel IRG methods with the following selections of stepsize rules: {\em backtracking stepsize}, {\em constant stepsize}, and {\em diminishing stepsize}. Address first an IRG method with the {\em backtracking line search}. Choose a line search scalar $\beta\in(0,1)$, a reduction factor $\gamma\in(0,1)$, and an artificial stepsize of null iterations $\tau\in(0,1)$. Consider the {\em Master Algorithm}~\ref{al gd inexact} with the stepsize sequence $\set{t_k}$ in Step~3 calculated as follows. If $d^k=0$, then put $t_k:=\tau$. Otherwise, we set
\begin{align}\label{defi: tk}
t_k:=\max\big\{t\;\big|\;f(x^k+td^k)\le f(x^k)-\beta t\big\|d^k\big\|^2,\;t=1,\;\gamma,\;\gamma^2,\ldots\big\}.
\end{align}

The next proposition shows that the stepsize sequence $\set{t_k}$  is well-defined.

\begin{Proposition}\label{descent} If $d^k\ne 0$, then there exists a positive number $\bar t$ such that 
\begin{align*}
f(x^k+td^k)\le f(x^k)-\beta t\big\|d^k\big\|^2\;\text{ for all }\;t\in [0,\bar t],
\end{align*}
which ensures the existence of the stepsize sequence $\{t_k\}$ in \eqref{defi: tk}. 
\end{Proposition}
\begin{proof} Suppose that $d^k\ne 0$ and get by Proposition~\ref{prop sufi holds} that $\dotproduct{\nabla f(x^k),d^k}\le-\norm{d^k}^2$. By the differentiability of $f$ at $x^k$, for each $t>0$ sufficiently small we have
\begin{align*}
f(x^k+td^k)-f(x^k)&=t\dotproduct{\nabla f(x^k),d^k}+o(t)\le-t\norm{d^k}^2+o(t)\\
&=-\beta t\norm{d^k}^2+t\brac{
(\beta-1)\norm{d^k}^2+ \dfrac{o(t)}{t}}.
\end{align*}
Since $o(t)/t\rightarrow0$ as $t\dn 0$ and since $(\beta-1)\norm{d^k}^2<0$, there exists $\bar t>0$ such that 
\begin{align*}
f(x^k+td^k)\le f(x^k)-\beta t\norm{d^k}^2\;\text{ for all }\;t\in(0,\bar t].
\end{align*}
Therefore, all the numbers $t_k$ introduced in (\ref{defi: tk}) are well-defined.
\end{proof}

Now we are ready to establish a major result about the {\em stationarity of accumulation points} of the iterative sequence generated by Algorithm~\ref{al gd inexact} with the backtracking line search.

\begin{Theorem}\label{convergence 1}
Let $\set{x^k}$ be the sequence of iterations generated by Algorithm~{\rm\ref{al gd inexact}} with the stepsize sequences $t_k$ from \eqref{defi: tk} of the backtracking line search. Assume that $\inf f(x^k)>-\infty$ and that $\rho_k\rightarrow0$ as $k\rightarrow\infty$. Then the following assertions hold:
\begin{itemize}
\item[\bf(i)] $\varepsilon_k\downarrow0$ and $r_k\downarrow 0$ as $k\to\infty$.

\item[\bf(ii)] Every accumulation point of $\set{x^k}$ is a stationary point of $f$.

\item[\bf(iii)] If the sequence $\set{x^k}$ is bounded, then the set of accumulation points of $\set{x^k}$ is nonempty, compact, and connected.

\item[\bf(iv)] If $\set{x^k}$ has an isolated accumulation point, then the entire sequence $\set{x^k}$ converges to this point.
\end{itemize}
\end{Theorem}
\begin{proof} From the choice of $\set{t_k}$ and Step~4 in Algorithm~\ref{al gd inexact}, for every $k\in\N$ we have
\begin{align}\label{tk dk}
\beta t_k\norm{d^k}^2\le f(x^k)-f(x^{k+1}).
\end{align}
Since $\inf f(x^k)>-\infty$, summing up  on both sides of (\ref{tk dk}) over $k=1,2,\ldots$ and using the relation $x^{k+1}=x^k+t_kd^k$, we get that
\begin{align}\label{taking sum 1}
\sum_{k=1}^\infty t_k\norm{d^k}^2<\infty\quad\text{ and }\quad \sum_{k=1}^\infty\norm{x^{k+1}-x^k}\cdot\norm{d^k}<\infty.
\end{align}

To verify assertion (i), recall by Proposition~\ref{null properties}(ii) that the convergence $\varepsilon_k\downarrow0$ is equivalent to $r_k\downarrow 0$ and to the fact that the set of null iterations $\mathcal{N}$ is infinite. Assume on the contrary that $\mathcal{N}$ is finite. By Proposition~\ref{null properties}(iii) with $N=\max\mathcal{N}+1$, we have
\begin{align}\label{dk bounded}
\norm{g^k}>r_N+\varepsilon_N\;\text{ and }\;\norm{d^k}>  r_N\;\text{ for all }\;k\ge N.
\end{align}
Then (\ref{taking sum 1}) gives us $\sum_{k=1}^\infty t_k<\infty$ and thus $t_k\downarrow0$ as $k\to\infty$. Choosing a larger number $N$ if necessary, we get that $t_k<1$ for all $k\ge N$. For such $k$, it follows from the exit condition of the algorithm that
\begin{align}\label{exit inexact}
-\gamma^{-1}\beta t_k\norm{d^k}^2< f(x^{k}+\gamma^{-1}t_kd^k)-f(x^k).
\end{align}
By the classical mean value theorem, there exists some $\tilde{x}^k\in[x^k,x^k+\gamma^{-1}t_kd^k]$ such that 
\begin{align*}
f(x^{k}+\gamma^{-1}t_kd^k)-f(x^k)= \gamma^{-1}t_k\dotproduct{d^k,\nabla f(\tilde{x}^k)}.
\end{align*}
Combining the latter with (\ref{exit inexact}) tells us that
\begin{align}\label{contra 1}
\dotproduct{-d^k,\nabla f(\tilde{x}^k)}\le\beta \norm{d^k}^2\;\text{ for all }\;k\ge N.
\end{align}
Using (\ref{taking sum 1}) and (\ref{dk bounded}), we have $\sum_{k=1}^\infty\norm{x^{k+1}-x^k}<\infty$, and thus $\set{x^k}$ converges to some $\bar x\in\R^n$. The continuity of $\nabla f$ ensures that $\nabla f(x^k)\rightarrow\nabla f(\bar x)$. Then employing $\norm{g^k-\nabla f(x^k)}\le \rho_k\rightarrow0$ yields $g^k\rightarrow\nabla f(\bar x)$ as $k\rightarrow\infty$.
It follows from Step~\ref{Step 2} that
\begin{align*}
-d^k&=\dfrac{\norm{g^k}-\varepsilon_k}{\norm{g^k}}g^k=\dfrac{\norm{g^k}-\varepsilon_N}{\norm{g^k}}g^k\;\text{ for all \;}k\ge N.
\end{align*}
Letting $k\rightarrow\infty$ leads us to the equalities
\begin{align}\label{1.8 1}
-d^k\rightarrow\bar g:=\dfrac{\norm{\nabla f(\bar x)}- \varepsilon_N }{\norm{\nabla f(\bar x)}}\nabla f(\bar x)={\rm Proj}\big(0,\mathbb{B}(\nabla f(\bar x),\varepsilon_N)\big).
\end{align}
Using $t_k\downarrow0$, we get that $\tilde{x}^k\rightarrow\bar x$, and thus $\nabla f(\tilde{x}^k)\rightarrow\nabla f(\bar x)$ as $k\to\infty$. Combining the latter with (\ref{contra 1}), (\ref{1.8 1}), and the projection characterization in \eqref{proj} verifies the estimates
\begin{align}
\norm{\bar g}^2\le\dotproduct{\bar g,\nabla f(\bar x)}\le\beta \norm{\bar g}^2.
\end{align}
This tells us that $\bar g=0$, which contradicts the condition $\norm{\bar g}\ge r_N$ by (\ref{dk bounded}). Therefore, we arrive at $\varepsilon_k\downarrow 0$ and $r_k\downarrow 0$ as $k\to\infty$, which completes the proof of assertion (i).

To justify assertions (ii)--(iv), recall from Corollary~\ref{overall} that $\set{d^k}$ is \textit{gradient associated} with $\set{x^k}$. Since $\varepsilon_k\downarrow 0$, we deduce from Proposition~\ref{goes to 0} that 
$0$ is an accumulation point of $\set{d^k}$. Combining these facts with (\ref{taking sum 1}) and $t_k\le 1$ whenever $k\in\N$ ensures that all the assumptions of Theorem~\ref{stationary point theorem} are satisfied. Therefore, we verify assertions (ii)--(iv) and finish the proof of the theorem.
\end{proof}

Next we consider problem (\ref{optim prob}) with the  objective function $f$ satisfying the \textit{$L$-descent condition} for some $L>0$. The following result establishes convergence properties of IRG method, which use either {\em diminishing} or {\em constant stepsizes}.

\begin{Theorem} \label{convergence constant}
Let $\set{x^k}$ be the sequence generated by Algorithm~{\rm\ref{al gd inexact}}, where 
\begin{itemize}
\item[\bf(a)]$f$ satisfies the $L$-descent condition;
\item[\bf(b)] either $\set{t_k}$ is diminishing, i.e.,
\begin{align}\label{diminishing}
t_k\downarrow0\;\mbox{ as }\;k\to\infty\;\text{ and }\quad\sum_{k=1}^\infty t_k=\infty,
\end{align}
or there exist $\delta,\delta'>0$ such that $\delta'\le\dfrac{2-\delta}{L}$ and
\begin{align}\label{constant}
t_k\in\sbrac{\delta',\dfrac{2-\delta}{L}}\;\text{ for all }\;k\in\N.
\end{align}
\end{itemize}
Assume that $\rho_k=\varepsilon_k$ for all $k\in\N$ and that $\inf f(x^k)>-\infty$. Then all the conclusions of Theorem~{\rm\ref{convergence 1}} hold. Moreover, if $\set{t_k}$ is chosen as \eqref{constant}, then $\nabla f(x^k)\rightarrow 0$ as $k\to\infty$.
\end{Theorem}

\begin{proof} We know from Remark~\ref{overall} that the direction sequence $\set{d^k}$ is \textit{gradient associated} with $\set{x^k}$. Furthermore, Proposition~\ref{prop sufi holds} tells us that $\set{d^k}$ satisfies the \textit{sufficient descent} condition (\ref{(b)}) with the constant $\kappa=1$. Note that if $\set{t_k}$ is chosen as either (\ref{diminishing}) or (\ref{constant}), then we always get that
\begin{align*}
\displaystyle{\sum_{k=1}^\infty}t_k=\infty\quad \text{and}\quad t_k\le\dfrac{2-\delta}{L}\;\text{ for sufficiently large }\;k\in\N.
\end{align*}
Combining these facts with the imposed $L$-descent condition on $f$ yields the fulfillment of assumptions (a), (b), (c) in Corollary~\ref{convergence diminishing}. Therefore, conclusions (ii)--(iv) of Theorem~\ref{convergence 1} hold. 
The proof of Corollary~\ref{convergence diminishing} also ensures that $0$ is an accumulation point of $\set{d^k}$. Thus it follows from Proposition~\ref{goes to 0} that $\varepsilon_k\downarrow 0$. Using Proposition~\ref{null properties}(ii), we have $r_k\downarrow0$, which verifies conclusion (i) of Theorem~\ref{convergence 1}. If $\set{t_k}$ is chosen as (\ref{constant}), its boundedness away from $0$ is guaranteed, and so Corollary~\ref{convergence diminishing} yields $\nabla f(x^k)\rightarrow 0$ as $k\to\infty$ and thus completes the proof of the theorem.
\end{proof}

The final part of our convergence analysis of the proposed IRG methods applies the {\em KL property} to establishing the {\em global} convergence of the {\em entire sequence} of iterations to a {\em stationary point} with deriving {\em convergence rates}. We start with the following simple albeit useful lemma.

\begin{Lemma}\label{gradien 3 dk}
Let $\set{x^k}$ be the sequence generated by Algorithm~{\rm\ref{al gd inexact}} with $\theta<\mu$. Assume that  $\varepsilon_k\downarrow 0$ and $r_k\downarrow0$ as $k\to\infty$. Then there exists some $N\in\N$ such that
\begin{align}\label{equation grad 3k}
\norm{\nabla f(x^k)}\le 3\norm{d^k}\;\text{ for all }\;k\notin\mathcal{N},\;k\ge N,
\end{align}
where the set $\mathcal{N}$ is taken from Definition~{\rm\ref{nul iter}}.
\end{Lemma}
\begin{proof} It follows directly from the assumptions of the lemma that there exists a natural number $N$ such that $\varepsilon_k\le r_k$ for all $k\ge N$. By Step~\ref{Step 2} of the algorithm, for any $k\ge N$ with $k\notin\mathcal{N}$ we have $\norm{g^k}>r_k+\varepsilon_k$ with the direction $d^k$ calculated in (\ref{choose dk}). Thus for such $k$ we get the estimates
\begin{align}\label{66}
\norm{d^k}=\norm{g^k}-\varepsilon_k>r_k+\varepsilon_k-\varepsilon_k= r_k\ge\varepsilon_k.
\end{align}
It follows from (\ref{calculate approx grad}) in Step~\ref{Step 1} and from (\ref{66}) that
\begin{align*}
\norm{\nabla f(x^k)}\le \norm{g^k}+\varepsilon_k=\norm{d^k}+2\varepsilon_k\le 3\norm{d^k},
\end{align*}
which verifies the conclusion of the lemma.
\end{proof}

The following two theorems provide conditions ensuring the global convergence of iterative sequences generated by Algorithm~\ref{al gd inexact} with different stepsize selections to a stationary point of $f$. The first theorem concerns the IRG methods with the {\em backtracking} stepsize.

\begin{Theorem}\label{convergence under KL}
Let $\set{x^k}$ be the iterative sequence generated by Algorithm~{\rm\ref{al gd inexact}} with the backtracking linesearch such that $\theta<\mu$. Assume that $\rho_k\rightarrow0$ as $k\rightarrow\infty$, $\set{x^k}$ has an accumulation point $\bar x$, and $f$ satisfies the KL property at $\bar x$. Then $\bar x$ is a stationary point of $f$, and we have $x^k\rightarrow\bar x$ as $k\to\infty$.
\end{Theorem}
\begin{proof}
Since $\bar x$ is an accumulation point of $\set{x^k}$, we can find some infinite set $J\subset\N$ such that $x^k\overset{J}{\rightarrow}\bar x$. It follows from the choice of $\set{t_k}$ in \eqref{defi: tk} that $\set{f(x^k)}$ is nonincreasing, which implies that 
\begin{align*}
\inf_{k\in \N} f(x^k)= \inf_{k\in J} f(x^k)=f(\bar x)>-\infty.
\end{align*}
Therefore, the results of Theorem~\ref{convergence 1} tell us that $\bar x$ is a stationary point of $f$ and that $\varepsilon_k\dn 0$, $r_k\downarrow 0$ as $k\to\infty$. We employ Proposition~\ref{general convergence under KL} to verify that $x^k\rightarrow\bar x$ along the entire sequence of iterations. Indeed, the imposed assumptions and the convergence $\varepsilon_k\dn 0$, $r_k\downarrow 0$ as $k\to\infty$ guarantee that all the requirements  of Lemma~\ref{gradien 3 dk} are satisfied. Pick $N\in\N$ such that (\ref{equation grad 3k}) holds. The choice of $\set{t_k}$ in (\ref{defi: tk}) ensures the lower estimate
\begin{align}\label{5.133}
f(x^{k})-f(x^{k+1})&\ge\beta t_k\big\|d^k\big\|^2\;\text{ for all }\;k\in\N.
\end{align}
Combining this with (\ref{equation grad 3k}) and the relation $x^{k+1}=x^k+t_kd^k$ yields
\begin{align}\label{5.13}
f(x^{k})-f(x^{k+1})\ge\dfrac{\beta}{3}\norm{\nabla f(x^k)}\cdot\norm{x^{k+1}-x^k}\;\text{ for all }\;k\notin\mathcal{N},\;k\ge N.
\end{align}
Observe that when $k\in\mathcal{N}$, both sides of
(\ref{5.13}) reduce to zero, and so (\ref{5.13}) is satisfied. Therefore, assumption (H1) in Proposition~\ref{general convergence under KL} holds. Moreover, for $k\ge N$ the conditions $f(x^{k+1})=f(x^k)$ and (\ref{5.133}) imply that $d^k=0$, and hence $x^{k+1}=x^k$. Thus assumption (H2) in  Proposition~\ref{general convergence under KL} is satisfied as well. Applying the latter proposition, we arrive at $x^k\rightarrow\bar x$ as $k\to\infty$ and complete the proof.
\end{proof}

The second theorem of the above type addresses the IRG methods with {\em diminishing} and {\em constant} selections of the stepsize sequence $\set{t_k}$.

\begin{Theorem}\label{convergence under KL 2}
Let the objective function $f$ satisfy the $L$-descent condition \eqref{descent condition} for some $L>0$, and let $\set{x^k}$ be the sequence generated by Algorithm~{\rm\ref{al gd inexact}} with $\theta<\mu$, $\rho_k=\varepsilon_k$ for all $k\in\N$, and either diminishing \eqref{diminishing} or constant stepsizes \eqref{constant}. Assume in addition that $\bar x$ is an accumulation point of the sequence $\set{x^k}$ and that $f$ satisfies the KL property at $\bar x$. Then $\bar x$ is a stationary point of $f$, and we have the convergence $x^k\rightarrow\bar x$ as $k\to\infty$.
\end{Theorem}
\begin{proof} Observe first that the assumptions imposed here yield those in Theorem~\ref{convergence constant} and Corollary~\ref{convergence diminishing} but $\inf f(x^k)>-\infty$. Similarly to the proof of Corollary~\ref{convergence diminishing}, we can show that $\set{f(x^k)}_{k\ge K}$ is nonincreasing for some $K\in\N.$ Since $\bar x$ is an accumulation point of $\set{x^k}$, similarly to the proof of Theorem~\ref{convergence under KL}, we deduce that $\inf f(x^k)=f(\bar x)>-\infty$, which verifies 
 the remaining assumption. Therefore, $\bar x$ is a stationary point of $f$ and $\varepsilon_k\dn 0$, $r_k\downarrow0$ as $k\to\infty$. The latter convergence together with the imposed assumptions guarantee the fulfillment of all the conditions of Lemma~\ref{gradien 3 dk}. Let $N\in\N$ be such that (\ref{equation grad 3k}) holds. Let $\delta>0$ be the constant given in (\ref{constant}). From the proof of Corollary~\ref{convergence diminishing}, we find some $N_1\ge N$ such that 
\begin{align}\label{latt}
f(x^{k})-f(x^{k+1})&\ge \dfrac{\delta}{2}t_k\norm{d^k}^2\;\text{ for all }\;k\ge N_1.
\end{align}
The relation $x^{k+1}=x^k+t_kd^k$ and (\ref{equation grad 3k}), (\ref{latt}) tell us that
\begin{align}\label{5.13 2}
f(x^{k})-f(x^{k+1})&\ge \dfrac{\delta}{6}\norm{\nabla f(x^k)}\norm{x^{k+1}-x^k}\;\text{ whenever }\;k\notin \mathcal{N},\;k\ge N_1.
\end{align}
Similarly to the proof of Theorem~\ref{convergence under KL}, we get $x^k\rightarrow\bar x$ as $k\to\infty$ and thus complete the proof.
\end{proof}

We'll see below that the boundedness of stepsizes away from $0$ plays a crucial role in establishing the {\em rate of convergence} of the IRG methods. This property automatically holds for constant stepsizes while may fail for diminishing ones. The next proposition shows that the boundedness from below is satisfied for the backtracking stepsize selection provided that the gradient of the objective function is {\em locally Lipschitzian} around accumulation points of iterative sequence. Observe that this property is strictly weaker than the (global) Lipschitz continuous of $\nabla f$. Indeed, $\mathcal{C}^2$-smooth functions have locally Lipschitzian gradients but do not need to have a globally Lipschitzian one as, e.g., for $f(x):=x^4$.

\begin{Proposition}\label{bounded away}
Let $\set{x^k}$ be a sequence generated by Algorithm~{\rm\ref{al gd inexact}} with the backtracking stepsize. Suppose that $\rho_k\rightarrow0$ as $k\rightarrow\infty$ and that there exists an infinite set $J\subset\N$ such that $\set{x^k}_{k\in J}$ converges to some $\bar x$. Assume further that $\nabla f$ is locally Lipschitzian around $\bar x$. Then the stepsize sequence $\set{t_k}_{k\in J}$ is bounded away from zero.
\end{Proposition}

\begin{proof}
 Assume on the contrary that $\set{t_k}_{k\in J}$ is not bounded away from zero. Then we find an infinite set $\bar J\subset J$ such that $t_k\overset{\bar J}{\longrightarrow}0$. Let $\tau\in(0,1)$ be an artificial stepsize of null iterations. Since $t_k\overset{\bar J}{\longrightarrow}0$, there exists a number $N\in\N$ such that 
 \begin{align}\label{tk < tau}
t_k<\tau<1\;\text{ for all }\;k\ge N,\;k\in\bar J.
 \end{align}
 This means that $k\notin\mathcal{N}$ whenever $k\ge N,\;k\in\bar J$. By Proposition~\ref{null properties}(i), we have $d^k\ne 0$ for all $k\ge N,\;k\in\bar J$. Then condition (\ref{choose dk}) in Step~\ref{Step 2} leads us to
\begin{align}\label{norm dk < norm gk}
\norm{d^k}=\norm{g^k}-\varepsilon_k\le\norm{g^k}\;\mbox{ for all }\;k\ge N,\;k\in\bar J.
\end{align}
Since $x^k\overset{\bar J}{\rightarrow}\bar x$, the continuity of $\nabla f$ and the estimate $\norm{\nabla f(x^k)-g^k}\le \rho_k\rightarrow0$ yield that $g^k\overset{\bar J}{\rightarrow}\nabla f(\bar x)$. Using (\ref{norm dk < norm gk}), we get that the sequence $\set{d^k}_{k\in\bar J}$ is bounded, and thus 
\begin{align}\label{bound of norm dk}
x^k+\gamma^{-1} t_kd^k\rightarrow\bar x\;\mbox{ as }\;k\overset{\bar J}{\rightarrow}\infty.
\end{align}
Since $\nabla f$ is locally Lipschitzian around $\bar x$, there exists a positive number $\delta$ such that $\nabla f$ is Lipschitz continuous on $\mathbb{B}(\bar x,\delta)$ with some modulus $L>0$. By (\ref{bound of norm dk}) and $x^k\overset{\bar J}{\rightarrow}\bar x$, we find $N_1\ge N$ with $x^k,x^k+\gamma^{-1}t_kd^k\in\mathbb{B}(\bar x,\delta)$ for all $k\ge N_1,\;k\in\bar J$. The Lipschitz continuity of $\nabla f$ on $\mathbb{B}(\bar x,\delta)$ with modulus $L$ yields by \cite[Proposition~A.24]{bertsekasbook} the $L$-descent condition, i.e.,
\begin{align}\label{lipschitz}
f(x)-f(y)\le\dotproduct{\nabla f(y),x-y}+\dfrac{L}{2}\norm{x-y}^2\;\text{ for all }\;x,y\in \mathbb{B}(\bar x,\delta).
\end{align}
Fixing $k\in\bar J,\;k\ge N_1$, we deduce from the above that $d^k\ne 0$, and $t_k<1$. The exit condition for the backtracking line search implies that
\begin{align}\label{exit 2}
-\gamma^{-1}\beta t_k\norm{d^k}^2<f(x^{k}+\gamma^{-1}t_kd^k)-f(x^k).
\end{align}
Applying (\ref{lipschitz}) for $x=x^{k}+\gamma^{-1}t_kd^k$ and $y=x^k$ we have that
\begin{align*}
f(x^{k}+\gamma^{-1}t_kd^k)-f(x^k)\le\gamma^{-1}t_k\dotproduct{\nabla f(x^k),d^k}+\dfrac{L\gamma^{-2}t_k^2}{2}\norm{d^k}^2.
\end{align*}
Combining this with (\ref{exit 2}) leads us to 
\begin{align*}
-\gamma^{-1}\beta t_k\norm{d^k}^2<\gamma^{-1}t_k\dotproduct{\nabla f(x^k),d^k}+\dfrac{L\gamma^{-2}t_k^2}{2}\norm{d^k}^2,
\end{align*}
or equivalently to the inequality 
\begin{align}\label{0< gamma-1}
0<\gamma^{-1}\beta t_k\norm{d^k}^2+\gamma^{-1}t_k\dotproduct{\nabla f(x^k),d^k}+\dfrac{L\gamma^{-2}t_k^2}{2}\norm{d^k}^2.
\end{align}
Proposition~\ref{prop sufi holds} and $d^k\ne 0$ tell us that $0<\norm{d^k}^2\le\dotproduct{\nabla f(x^k),-d^k}$. Then we deduce from (\ref{0< gamma-1}) the fulfillment of the estimate
\begin{align*}
0<\gamma^{-1}\beta t_k\dotproduct{\nabla f(x^k),-d^k}+\gamma^{-1}t_k\dotproduct{\nabla f(x^k),d^k}+\dfrac{L\gamma^{-2}t_k^2}{2}\dotproduct{\nabla f(x^k),-d^k}.
\end{align*}
Dividing both sides above by $\gamma^{-1}t_k\dotproduct{\nabla f(x^k),-d^k}>0$, we get $0<\beta-1+\dfrac{L\gamma^{-1}t_k}{2}$.
Letting $k\overset{\bar J}{\longrightarrow}\infty$ yields $\beta\ge 1$, which contradicts the choice of $\beta\in(0,1)$. Thus we verify that the sequence $\set{t_k}_{k\in J}$ is bounded away from zero, which completes the proof of the proposition.
\end{proof}

The last two theorems establish efficient conditions ensuring the {\em convergence rates} in Algorithm~\ref{al gd inexact} under different stepsize selections. Having the sequence of iterations $\set{x^k}$ generated by this algorithm, we obtain first from Proposition~\ref{null properties}(iii) that if $\N\setminus\mathcal{N}$ is finite, then $\set{x^k}$ stops after a finite number of iterations. Thus we consider the case where the set $\N\setminus\mathcal{N}$ is infinite and can be numerated as $\set{j_1,j_2,\ldots}$. Construct the sequence $\set{z^k}$ by
\begin{align}\label{z_k}
 z^k:=x^{j_k}\;\text{ for all }\;k\in\N.
\end{align}
We have $j_{k+1}\ge j_{k}+1$ whenever $k\in\N$. If the equality holds therein, then $z^{k+1}=z^{j_k+1}$. Otherwise, by taking into account that the indices $j_{k}+1,\ldots,j_{k+1}-1$ correspond to null iterations, we get that
\begin{align}\label{5 23}
x^{j_k+1}=x^{j_{k}+2}=\ldots=x^{j_{k+1}-1}=x^{j_{k+1}}=z^{k+1}.
\end{align}
Therefore, it follows from $j_k\notin\mathcal{N}$ that
\begin{align}\label{differ}
z^{k+1}=x^{j_k+1}\ne x^{j_k}=z^k\;\text{ for all }\;k\in\N.
\end{align}
Furthermore, Proposition~\ref{null properties}(iv) tells us that $\bar x$ is an accumulation point of $\set{z^k}$ if and only if $\bar x$ is also an accumulation point of $\set{x^k}$.\vspace*{0.05in}

The first theorem about the convergence rates concerns  Algorithm~\ref{al gd inexact} with the {\em backtracking stepsize}.

\begin{Theorem} \label{rate} 
Consider Algorithm~{\rm\ref{al gd inexact}} with the backtracking stepsize selections under the condition $\theta<\mu$. Let $\set{x^k}$ be the iterative sequence generated by this algorithm. Suppose that $\rho_k\rightarrow0$ as $k\rightarrow\infty$. Assume further that $\set{x^k}$ has an accumulation point $\bar x$, that $f$ satisfies the KL property at $\bar x$ with $\psi(t)=Mt^{q}$ for some $M>0$ and $q\in(0,1)$, and that $\nabla f$ is locally Lipschitzian around $\bar x$.  The following convergence rates are guaranteed for the sequence $\set{z^k}$ defined in \eqref{z_k}:
\begin{itemize}
\item[\bf(i)] If $q\in(0,1/2]$, then the sequence $\set{z^k}$ converges linearly to $\bar x$.
\item[\bf(ii)] If $q\in(1/2,1)$, then there exists a positive constant $\varrho$ such that 
\begin{align*}
 \norm{z^k-\bar x}\le\varrho k^{-\frac{1-q}{2q-1}}\;\text{ for all large }\;k\in\N.
\end{align*}
\end{itemize}
\end{Theorem}
\begin{proof}
The imposed assumptions yield the fulfillment of those in Theorem~\ref{convergence under KL}, and so lead us to the convergence $x^k\rightarrow\bar x$ as $k\to\infty$. Then the local Lipschitz continuity of $\nabla f$ around $\bar x$ and Proposition~\ref{bounded away} ensure that the sequence $\set{t_k}$ is bounded away from zero.
 
 To deduce now the claimed convergence rates in (i)--(iii) from Theorem~\ref{general rate}, define $\tau_k:=t_{j_k}$ for all $k\in\N$. Then $\set{\tau_k}$ is also bounded away from zero as a subsequence of $\set{t_k}$. Furthermore, using (\ref{5 23}) and the line search conditions, we have
\begin{align}\label{beta tau k}
 f(z^k)-f(z^{k+1})&=f(x^{j_{k}})-f(x^{j_{k}+1})\ge\beta t_{j_k}\norm{d^{j_k}}^2\nonumber\\
 &=\dfrac{\beta}{t_{j_k}}\norm{x^{j_k+1}-x^{j_k}}^2=\dfrac{\beta}{\tau_k}\norm{z^{k+1}-z^k}^2
\end{align}
for all $k\in\N$. Note that all the assumptions of Theorem~\ref{convergence 1} are satisfied, and so Lemma~\ref{gradien 3 dk} holds. Pick any $N\in\N$ from (\ref{equation grad 3k}) and fix $k\ge N$. Then using (\ref{5 23}) and (\ref{equation grad 3k}) with taking into account that $j_k\notin\mathcal{N}$ for $j_k\ge k$ leads us to
\begin{align*}
\norm{\nabla f(z^k)}=\norm{\nabla f(x^{j_k})}&\le \dfrac{3}{t_{j_k}}\norm{x^{j_k+1}-x^{j_k}}=\dfrac{3}{\tau_k}\norm{z^{k+1}-z^k}.
\end{align*}
Apply finally Theorem~\ref{general rate} to $\set{z^k}$ and $\set{\tau_k}$ while remembering that $z^{k+1}\ne z^k$ for all $k\in\N$ from (\ref{differ}). This verifies the convergence rates (i)--(iii) claimed in the theorem.
\end{proof}

The next theorem on the convergence rates addresses Algorithm~\ref{al gd inexact} with the {\em constant stepsizes}.

\begin{Theorem}\label{rate constant}
Let $f$ satisfy the $L$-descent condition for some $L>0$, and let $\set{x^k}$ be the iterative sequence generated by Algorithm~{\rm\ref{al gd inexact}} with the constant stepsizes \eqref{constant} under the condition $\theta<\mu$ and the selection $\rho_k=\varepsilon_k$ for all $k\in\N$. Suppose that $\set{x^k}$ has an accumulation point $\bar x$ and that $f$ satisfies the KL property at $\bar x$ with $\psi(t)=Mt^{q}$ for some $M>0$ and $q\in(0,1)$. Then the following convergence rates are guaranteed for the iterative  sequence $\set{z^k}$ defined in \eqref{z_k}:
\begin{itemize}
\item[\bf(i)] If $q\in(0,1/2]$, then the sequence $\set{z^k}$ converges linearly to $\bar x$.
\item[\bf(ii)] If $q\in(1/2,1)$, then there exists a positive constant $\varrho$ such that 
\begin{align*}
\norm{z^k-\bar x}\le\varrho k^{-\frac{1-q}{2q-1}}\;\text{ for all large }\;k\in\N.
\end{align*}
\end{itemize}
\end{Theorem}
\begin{proof}
 Note that our assumptions yield the fulfillment of those in 
Theorem~\ref{convergence under KL 2}, and thus we have that $x^k\rightarrow\bar x$ as $k\to\infty$. Defining $\tau_k:=t_{j_k}$ for all $k\in\N$ ensures that the stepsize sequence $\set{\tau_k}$ is bounded away from zero. Note that all the assumptions in Corollary~\ref{convergence diminishing} hold, and let $\delta>0$ be the constant taken from in (\ref{constant}). By the $L$-descent property of $f$ and the constant stepsize selection, we find by arguing similarly to the proof of Corollary~\ref{convergence diminishing} a number 
$ N\in\N$ such that 
\begin{align}
f(x^{k})-f(x^{k+1})&\ge\dfrac{\delta}{2}t_k\norm{d^k}^2
\;\text{ whenever }\;k\ge N.
\end{align}
Since $j_k\ge k\ge N$ for such $k$, it follows that 
 \begin{align*}\label{beta tau k}
f(z^k)-f(z^{k+1})&=f(x^{j_{k}})-f(x^{j_{k}+1})\ge \dfrac{\delta}{2} t_{j_k} \norm{d^{j_k}}^2\\
&=\dfrac{\delta}{2t_{j_k}}\norm{x^{j_k+1}-x^{j_k}}^2=\dfrac{\delta}{2\tau_k}\norm{z^{k+1}-z^k}^2.
\end{align*}
Note that all the assumptions of Theorem~\ref{convergence constant} are satisfied. Using this result together with Lemma \ref{gradien 3 dk} and then arguing as in the proof of Theorem~\ref{rate}, we complete the proof of this theorem.
\end{proof}\vspace*{-0.2in}

\section{Applications and Numerical Experiments}\label{sec: Numerical experiments}\vspace*{-0.05in}

In this section, we present efficient implementations of the developed IRG methods to solving particular classes of optimization problems that appear in practical modeling. We conduct numerical experiments and compare the results of computations by using our algorithms with those obtained by applying some other well-known methods. This section is split into two subsections addressing different classes of problems with the usage of different algorithms.\vspace*{-0.1in}

\subsection{Comparison with classical inexact  proximal point method}\vspace*{-0.05in}
 
This subsection addresses the {\em Least Absolute Deviations $($LAD$)$ Curve-Fitting problem} which is formulated as follows:
\begin{equation}\label{LAD}
\mbox{minimize }\;g(x):=\norm{Ax-b}_1\;\mbox{ over }\;x\in\R^n,
\end{equation}
where $A$ is an $m\times n$ matrix, $b$ is a vector in $\R^m$, and $\norm{u}_1:=\sum_{k=1}^m|u_k|$ for any $u=(u_1,\ldots,u_m)\in\R^m$.
Problem \eqref{LAD} exhibits robustness in outliers resistance and appears in many applied areas; see, e.g., \cite{peter80} for more discussions. Observe that \eqref{LAD} is a problem of {\em nonsmooth convex optimization}, but we can reduce it to a smooth problem by using a regularization procedure. In this way, we solve \eqref{LAD} by using our {\em IRG method with constant stepsize} and compare our approach with the usage of the {\em inexact proximal point method} (IPPM) proposed by Rockafellar in \cite{rockafellar76}.\vspace*{0.05in} 

To proceed, recall that the {\it Moreau envelope} and the {\it proximal mapping} of $g$ are defined by
\begin{equation}\label{Moreau}
e_{g} (x):=\inf_{y \in \R^n}\ph_{x}(y)\;\mbox{ and }\;\prox_{ g }(x):=\mathop{\rm argmin}_{y\in \R^n}\ph_{x}(y),\quad x\in\R^n,
\end{equation}
where the minimization mapping $\ph_{x}:\R^n\rightarrow\R$ is given by
\begin{align}\label{phi lambda}
\ph_{x}(y):=g (y)+\dfrac{1}{2}\norm{y-x}^2,\quad y\in\R^n.
\end{align}
Since $g$ is convex, it follows from \cite[Propositions~12.28 and 12.30]{bauschkebook} that $e_g $ is $\mathcal{C}^1$-smooth and that its gradient is Lipschitz continuous with constant $1$ being represented by
\begin{equation}\label{gradientMoreau}
\nabla e_g (x)=x-\text{\rm Prox}_{g}(x)\;\mbox{ for all }\;x\in \R^n.
\end{equation}
Moreover, the set of minimizers of $g$ coincides with the set of zeros of the gradient mapping $\nabla e_g$.
 
 This tells us that problem \eqref{LAD} can be equivalently transformed into the problem of finding stationary points of the smooth function $f:=e_g$. Therefore, it is possible to solve \eqref{LAD} by using Algorithm~\ref{al gd inexact} with constant stepsize, where an inexact gradient $g^k$ of $\nabla f(x^k)$ in Step~\ref{Step 1} satisfying the simplified condition \eqref{IRG condition constant} can be chosen from the conditions
\begin{align}\label{ineq IRG}
g^k:=x^k-p^k\;\text{ with }\norm{p^k-\prox_g(x^k)}\le \varepsilon_k.
\end{align}
Meanwhile, the iterative procedure of IPPM in \cite[page~878]{rockafellar76} for solving \eqref{LAD} is given by
\begin{align}\label{ineq lambda eps}
x^{k+1}=p^k\text{ with }\norm{p^k-\prox_g(x^k)}\le \delta_k,\text{ where }\sum_{k=1}^\infty \delta_k<\infty.
\end{align}
Since the function $\ph_{x^k}$ in \eqref{phi lambda} is strongly convex with constant $1$ \cite[Definition~2.1.3]{nesterovbook18}, the error bound for the distance between the inexact proximal point $p^k$ and the exact one $\prox_g(x^k)$ in \eqref{ineq IRG} and \eqref{ineq lambda eps} is satisfied if 
\begin{align}\label{def omega}
\ph_{x^k}(p^k)\le \inf\ph_{x^k}+\omega_k,
\end{align}
where $\omega_k:=\dfrac{\varepsilon_k^2}{2}$ for \eqref{ineq IRG} and $\omega_k:=\dfrac{\delta_k^2}{2}$ for \eqref{ineq lambda eps} by using \cite[Theorem~2.1.8]{nesterovbook18}. In this numerical experiment, we run the {\em Fast Iterative Shrinkage-Thresholding Algorithm} (FISTA) of Beck and Teboulle \cite{beck09} for the dual function of $\ph_{x^k}$ until the duality gap is below $\omega_k$, which therefore ensures \eqref{def omega}.\vspace*{0.05in}

The initial points are chosen as $x^1:=0_{\R^n}$ for both algorithms, while the detailed settings of each algorithm are given as follows:
\begin{itemize}\vspace*{-0.05in}
\item IRG: $\varepsilon_1=10, \theta=\mu=0.5$. Two selections of the initial radius $r_1$ are $20$ and $5$, which correspond to versions IRG-20 and IRG-5, respectively. To simplify the iterative sequence of Algorithm~\ref{al gd inexact} when $\norm{g^k}\le r_k+\varepsilon_k$, we put $x^{k+1}:=p^k$, which corresponds to the choice of stepsize $t_k=\frac{\norm{g^k}}{\norm{g^k}-\varepsilon_k}$.\vspace*{-0.1in}

\item IPPM: $\omega_k=\dfrac{1}{k^p}$ for all $k\in\N$, where $p=4$ or $p=2.1.$ These selections together with the definition of $\omega_k$ in \eqref{def omega} ensure that $\sum_{k=1}^\infty\delta_k<\infty$ as required for IPPM in \eqref{ineq lambda eps}. We also use the labels IPPM-4 and IPPM-2.1 for these versions of IPPM, respectively.
\end{itemize}\vspace*{-0.05in}

In this numerical experiment, we let IPPM-2.1 run for 200 iterations and record the function value obtained by this method. Then other methods run until their function values are lower than the recorded one of IPPM-2.1. We stop the methods when the time reaches the limit of $4000$ seconds.  The data $A,b$ are generated randomly with i.i.d. (identically and independent distributed) standard Gaussian entries. To avoid algorithms from reaching the solution promptly, we consider only the cases where $m\le n$ in \eqref{LAD}.

The numerical experiment is conducted on a computer with 10th Gen Intel(R) Core(TM) i5-10400 (6-Core 12M Cache, 2.9GHz to 4.3GHz) and 16GB RAM memory. The codes are written in MATLAB R2021a.  The detail information for the results are presented in Table \ref{table LAD}, where 'Test \#', '\textit{iter}', '\textit{fval}', '\textit{time}' mean Test number, the number of iterations, value of the objective function at the last iteration and the computational time, respectively. The errors $\omega_k$ in the inexact proximal point calculations \eqref{def omega} and the function values obtained by the algorithms over the duration of time  are also graphically illustrated in Figure~\ref{fig:err3} and Figure~\ref{fig:fval4}.
\begin{table}[H]
\small
\begin{tabular}{|cll|lll|lll|lll|lll|} 
\hline
\multirow{2}{*}{Test \#} & \multicolumn{1}{c}{\multirow{2}{*}{m}} & \multicolumn{1}{c|}{\multirow{2}{*}{n}} & \multicolumn{3}{c|}{IPPM-2.1}     & \multicolumn{3}{c|}{IPPM-4}       & \multicolumn{3}{c|}{IRG-5}    & \multicolumn{3}{c|}{IRG-20}       \\ 
\cline{4-15}
& \multicolumn{1}{c}{}      & \multicolumn{1}{c|}{}      & \multicolumn{1}{c}{iter} & \multicolumn{1}{c}{fval} & \multicolumn{1}{c|}{time} & \multicolumn{1}{c}{iter} & \multicolumn{1}{c}{fval} & \multicolumn{1}{c|}{time} & \multicolumn{1}{c}{iter} & \multicolumn{1}{c}{fval} & \multicolumn{1}{c|}{time} & \multicolumn{1}{c}{iter} & \multicolumn{1}{c}{fval} & \multicolumn{1}{c|}{time}  \\ 
\hline
1 & 300& 300 & 200& 3.44        & 11.79        & 200& 3.44        & 100.63   & 209& 3.44        & \textbf{5.51}& 210& 3.43        & \textit{7.91} \\
2 & 300& 600 & 200& 0.00        & 1.56& 17& 0.00        & \textit{0.38}& 16& 0.00        & \textbf{0.33}& 19& 0.00        & 0.44\\
3 & 600& 600 & 200& 0.44        & 140.35       & 201& 0.44        & 863.39   & 210& 0.44        & \textbf{78.09} & 211& 0.44        & \textit{119.78} \\
4 & 600& 1200& 200& 0.00        & 5.26& 16& 0.00        & \textit{1.72}& 17& 0.00        & \textbf{1.42}& 20& 0.00        & 1.77\\
5 & 1200 & 1200& 67& 8.00        & 4000& 14& 14.76       & 4000& 75& 7.98    & \textbf{949.46}& 75& 8.00        & \textit{1672.52}\\
\hline
\end{tabular}
\caption{Results for LAD Curve-Fitting problem}
\label{table LAD}
\end{table}
\begin{figure}[H]
\centering
\includegraphics[width=.3\textwidth]{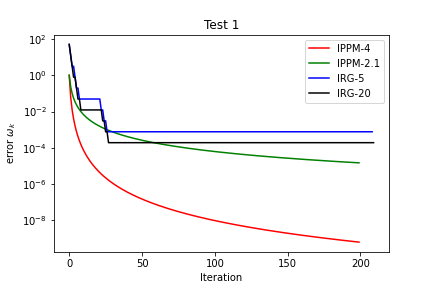}\quad
\includegraphics[width=.3\textwidth]{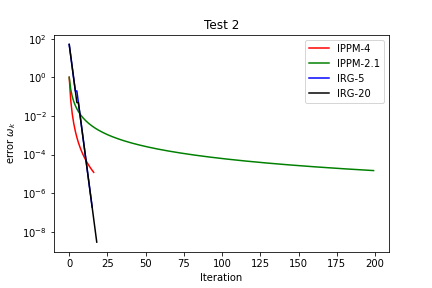}\quad
\includegraphics[width=.3\textwidth]{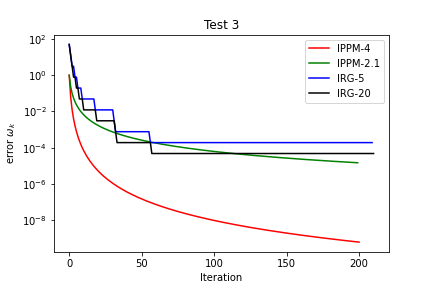}

\medskip

\includegraphics[width=.3\textwidth]{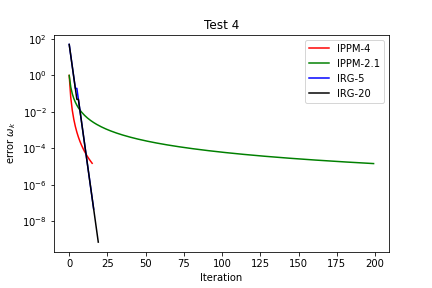}\quad
\includegraphics[width=.3\textwidth]{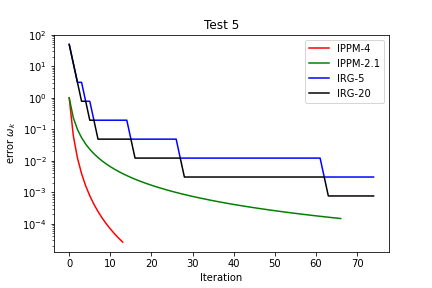}

\caption{Errors in proximal points calculation in iterations}
\label{fig:err3}
\end{figure}
 \begin{figure}[H]
\centering
\includegraphics[width=.3\textwidth]{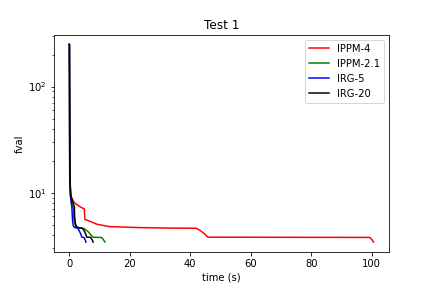}\quad
\includegraphics[width=.3\textwidth]{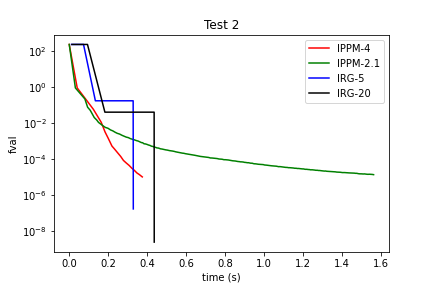}\quad
\includegraphics[width=.3\textwidth]{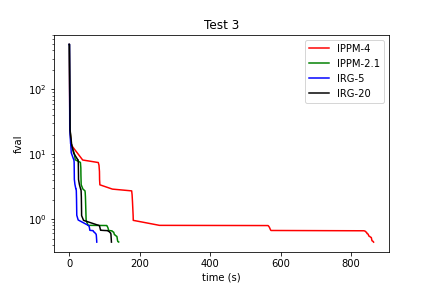}

\medskip

\includegraphics[width=.3\textwidth]{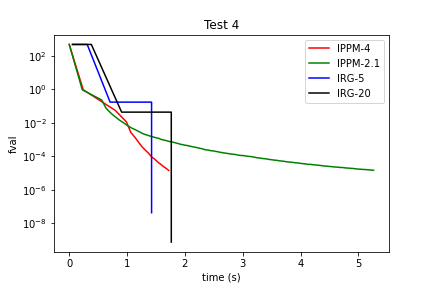}\quad
\includegraphics[width=.3\textwidth]{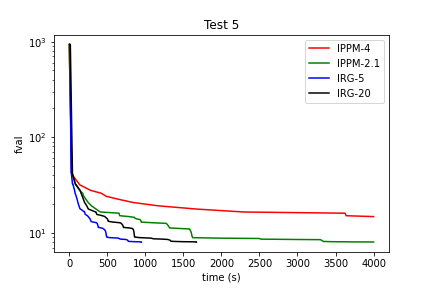}

\caption{Value of the objective function with respect to the  computational time}
\label{fig:fval4}
\end{figure}
It can be seen from Table~\ref{table LAD} that IRG-5 has the best performance in this numerical experiment. IRG-20 is the second fastest algorithm in Tests 1, 3, 5 while it is slightly slower than IPPM-4 in Tests 2, 4. In Test 5 with the largest dimensions $m=n=1200,$ IRG-5 is around 4 times faster than IPPM-2.1 while IPPM-4 even cannot reach the value obtained by IPPM-2.1 within the time limit. In this test, IRG-20 is also around 2.5 times faster than IPPM-2.1.

The graphs in Figure~\ref{fig:err3} and Figure~\ref{fig:fval4} show that the errors (in inexact proximal point calculations) of IRG are automatically adjusted to be suitable for different problems:
\begin{itemize}\vspace*{-0.05in}

\item In Tests 1, 3, 5 with $m=n$, it can be seen from Figure~\ref{fig:fval4} that IPPM-2.1 is faster than IPPM-4, which means that the use of larger errors is preferred in this case. Then Figure~\ref{fig:err3} shows that the errors used in IRG stagnate at most of the iterations. As a result, the IRG methods use the errors larger than that of the IPPM methods and thus achieve better performances. \vspace*{-0.05in}

\item In Tests 2, 4 with $m<n$, IPPM-4 with smaller errors performs better than IPPM-2.1. In this cases, the IRG methods decrease in almost every iteration and achieve smaller errors in comparison with IPPM-4. 
\end{itemize}\vspace*{-0.2in}

\subsection{Comparison with the exact gradient descent methods}\vspace*{-0.05in}

In the numerical experiments presented in this subsection, we show that our IRG method with backtracking stepsize, based on the usage of  inexact gradients, performs well compared with the famous methods employing the exact gradient calculation, which are the  {\em reduced gradient} (RG) method  method and {\em gradient descent} (GD) method in the following setting:
\begin{enumerate}\vspace*{-0.05in}
\item The {\em accuracy} of the  inexact gradient $g^k$ is {\em low}, i.e., $\norm{g^k-\nabla f(x^k)}\le\delta_k$, where $\delta_k$ is not too small relative to $\norm{\nabla f(x^k)}$.\vspace*{-0.05in}

\item The {\em accuracy} required for the solution is {\em increasing}.\vspace*{-0.05in}

\end{enumerate}
To demonstrate this, we choose the following two well-known smooth benchmark functions in global optimization taken from the survey paper \cite{jamil13}. 
\begin{itemize}\vspace*{-0.05in}
 \item The {\em Dixon $\&$ Price} function defined by
\begin{align*}
 f_{\rm dixon}(\mathbf{x}):=\left(x_{1}-1\right)^{2}+\sum_{i=2}^{n}i\left(2 x_{i}^{2}-x_{i-1}\right)^{2},\quad x\in\R^n.
\end{align*}The global minimum of this function is $\bar f_{\rm dixon}=0$, and the two solutions $x^*,y^*\in\R^n$ are 
\begin{align*}
\begin{cases}
x^*_1=1,&\\
x^*_{k}=\sqrt{\dfrac{x_{k-1}}{2}}&\text{ for }\;k=2,\ldots,n
\end{cases}
\end{align*}
and by $y^*_k=x^*_k$ for all $k=1,\ldots,n-1,$ $y^*_n=-x^*_n$.\vspace*{-0.05in}

\item The {\em Rosenbrock~{\rm 1}} function defined by
\begin{align*}
 f_{\rm rosen}(\mathbf{x}):=\sum_{i=1}^{n-1}\left[100\left(x_{i+1}-x_{i}^{2}\right)^{2}+\left(x_{i}-1\right)^{2}\right],\quad x\in\R^n.
\end{align*}
The global minimum of this function is $\bar f_{\rm rosen}=0$, and the unique solution is $(1,\ldots,1)\in\R^n$.
\end{itemize}
Since the information about the convexity and the Lipschitz continuity of gradients of the chosen objective functions is {\em unknown}, our experiments are conducted by algorithms, where stepsizes are obtained from the corresponding line searches. We use the following abbreviations:
\begin{itemize}\vspace*{-0.05in}
\item GD: {\em Gradient descent method with the backtracking linesearch}.\vspace*{-0.05in}

\item RGB \& IRGB: {\em Reduced gradient method with the backtracking linesearch $\&$ Inexact reduced gradient method with the backtracking linesearch}; see (\ref{defi: tk}).\vspace*{-0.05in}
\end{itemize}
To generate the inexactness for testing purposes, given the gradient error $\delta_k:=\min\set{\varepsilon_k,\rho_k}$ as in (\ref{calculate approx grad}), we create an inexact gradient $g^k$ by adding a random vector with the norm $0.5\delta_k$ to the exact gradient $\nabla f(x^k)$. To ensure  manually controlled errors between the exact gradients and inexact ones that do not decrease so fast, we choose $\rho_k:=1/\log (k+1)$. For all the methods in our experiments, the linesearch parameters are chosen as $\beta=0.7$ and $\gamma=0.5$. The initial radii $\varepsilon_1=r_1=5$ and the radius reduction factors  $\theta=0.7,\;\mu=0.7$ are also used for the RG and IRG methods. To avoid the initial points from being identical with  the solutions, we choose $x^1:=0_{\R^n}$ on tests using the Rosenbrock~1 function. In the tests using the Dixon \& Price functions, we choose $x^1:=1_{\R^n}$ to avoid the algorithms from going to different solutions. The condition
\begin{align*}
\norm{\nabla f(x)}\le\nu,\;\text{ where either }\;\nu=0.01 \text{ or }\nu=0.001.
\end{align*} 
is used as the stopping criterion for all the tests. The detailed information of the numerical experiments and the achieved numerical results are presented in Table~\ref{table 1}. The problem names are given in the forms Dn and Rn, where D stands for Dixon \& Price, R stands for Rosenbrock~1, and $n$ is the dimension of the tested problem. In these tables, ``Iter", ``fval" stand for the number of iterations and the function value at the last iteration. 
\begin{table}[H]
\centering
\begin{tabular}{|lc|ll|llll|lll|} 
\hline
\multicolumn{1}{|c}{\multirow{2}{*}{Name}} & \multirow{2}{*}{$\nu$} & \multicolumn{2}{c|}{GD}          & \multicolumn{4}{c|}{IRGB}  & \multicolumn{3}{c|}{RGB}      \\ 
\cline{3-11}
\multicolumn{1}{|c}{}  &      & \multicolumn{1}{c}{Iter} & \multicolumn{1}{c|}{fval} & \multicolumn{1}{c}{Iter} & \multicolumn{1}{c}{fval} & \multicolumn{1}{c}{$\varepsilon_k$} & \multicolumn{1}{c|}{$\delta_k$} & \multicolumn{1}{c}{Iter} & \multicolumn{1}{c}{fval} & \multicolumn{1}{c|}{$\varepsilon_k$}  \\ 
\hline
D200 & 0.01 & 1928 & 2.8E-05 & 1110 & 2.6E-05& 4.0E-03      & 2.0E-03& 998  & 2.76E-05      & 4.0E-03 \\
D500 & 0.01 & 3831 & 2.8E-05 & 4782 & 2.7E-05& 4.0E-03      & 2.0E-03& 5012 & 2.43E-05      & 4.0E-03 \\
D1000& 0.01 & 7655 & 2.8E-05 & 9347 & 2.6E-05& 4.0E-03      & 2.0E-03& 9271 & 2.58E-05      & 4.0E-03 \\
R200 & 0.01 & 20357& 9.5E-05 & 24966& 9.2E-05& 4.0E-03      & 2.0E-03& 25162& 8.65E-05      & 4.0E-03 \\
R500 & 0.01 & 46135& 9.3E-05 & 59768& 7.3E-05& 4.0E-03      & 2.0E-03& 59604& 9.20E-05      & 4.0E-03 \\
R1000& 0.01 & 89130& 9.4E-05 & 117754 & 8.8E-05& 4.0E-03      & 2.0E-03& 117845 & 9.09E-05      & 4.0E-03 \\
D200 & 0.001& 3294 & 2.7E-07 & 1870 & 2.7E-07& 4.7E-04      & 2.3E-04& 1704 & 2.8E-07& 4.7E-04 \\
D500 & 0.001& 6543 & 2.8E-07 & 8151 & 2.4E-07& 4.7E-04      & 2.3E-04& 7933 & 2.2E-07& 4.7E-04 \\
D1000& 0.001& 13078& 2.8E-07 & 15946& 2.7E-07& 4.7E-04      & 2.3E-04& 15598& 2.3E-07& 4.7E-04 \\
R200 & 0.001& 22664& 9.6E-07 & 26958& 8.9E-07& 4.7E-04      & 2.3E-04& 27395& 9.7E-07& 4.7E-04 \\
R500 & 0.001& 48442& 9.5E-07 & 61998& 7.8E-07& 4.7E-04      & 2.3E-04& 61875& 9.6E-07& 4.7E-04 \\
R1000& 0.001& 91431& 9.6E-07 & 119687 & 9.0E-07& 4.7E-04      & 2.3E-04& 120321 & 9.6E-07& 4.7E-04 \\
\hline
\end{tabular}
\caption{Comparison with  Dixon \& Price and Rosenbrock~1}
\label{table 1}
\end{table}

It can be seen that the performance of the IRG and RG methods in Tests~D200 with $\nu=0.01$ and $\nu=0.001$ is better than that of the GD method, while the latter is more efficient in the other tests. It is reasonable that GD usually performs better since it uses the exact gradient while RGB uses the reduced gradient and IRGB uses even the inexact one. In the worst case in Test R1000 with $\nu=0.01,$ the number of iterations of IRGB is equal around 1.3 times that of GD. It shows that IRGB does not suffer much from the use of inexact gradient compared with the performance of GD using the exact gradient. Table~\ref{table 1} also shows that the decrease of $\nu$ in $10$ times results in the increase of the number of iterations in IRGB with the rate at most $1.7$, where the worst case corresponds with the tests D500. This rate is similar to the rate obtained by the GD method in these tests, which confirms that our IRG method with the backtracking stepsize does not suffer from error accumulation.

The graphs below show that the errors $\delta_k$ of the inexact gradient used in IRGB is automatically adjusted to be not too small or too large compared with $\norm{\nabla f(x^k)}$. This confirms the intuitive conclusion on the IRG methods discussed in Remark~\ref{remark intu}(ii). Figure~\ref{fig:smooth} shows that the selections of errors $\delta_k=k^{-p},p\ge 1$ in the existing methods \cite{bertsekasbook,nesterov14,gilmore95} do not fit the unexpected fluctuations in the norm of the exact gradient given in the tests using the Rosenbrock function, which may lead to the over approximation and under approximation issues discussed in Section~\ref{sec:intro}.

\begin{figure}[H]
\centering
\includegraphics[width=.3\textwidth]{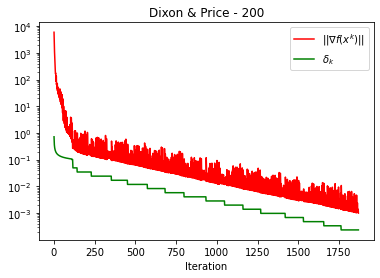}\quad
\includegraphics[width=.3\textwidth]{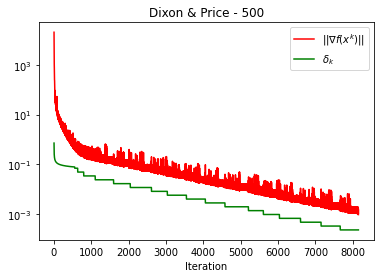}\quad
\includegraphics[width=.3\textwidth]{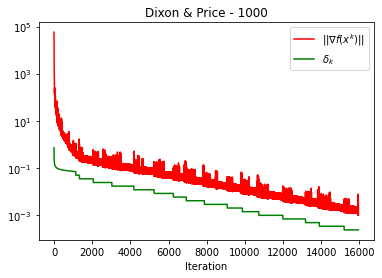}

\medskip

\includegraphics[width=.3\textwidth]{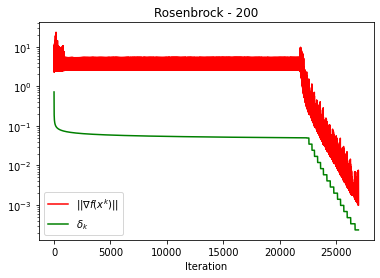}\quad
\includegraphics[width=.3\textwidth]{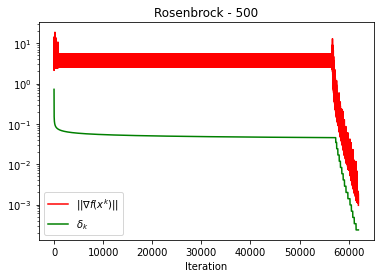}\quad
\includegraphics[width=.3\textwidth]{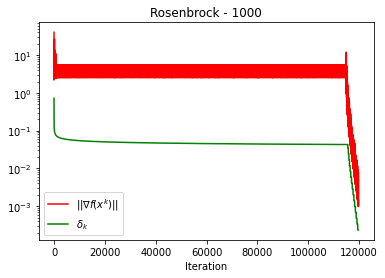}

\caption{Errors of IRGB compared with the norm of exact gradient}\label{fig:smooth}
\end{figure}\vspace*{-0.2in}

\section{ Conclusions and Further Research}\label{sec:conc}\vspace*{-0.05in}

In this paper, we propose and develop the inexact reduced gradient methods with different stepsize selections to solve problems of nonconvex optimization. These methods achieve stationary accumulation points and, under additional assumptions on the KL property of the objective functions, the global linear convergence. The convergence analysis of the developed algorithms is based on novel convergence results established for general linesearch methods. The theoretical and numerical comparisons show that our methods do not suffer much from the error accumulation and are able to automatically adjust the errors in the exact gradient approximations to get a better performance than the existing methods using common selections of errors.

In our future research, we aim at developing the IRG methods in different directions, which include designing zeroth-order algorithms by using practical methods for approximating gradients, designing inexact versions of methods frequently used in nonconvex nonsmooth optimization, e.g., the proximal point and proximal gradient methods, and also designing appropriate IRG methods for problems of constrained optimization. The obtained results would allow us to develop new applications to important classes of models in machine learning, statistics, and related disciplines.
\end{document}